\documentclass[11pt]{article}

 \usepackage{shapepar}
\usepackage{amsmath}
\usepackage{amsfonts}
\usepackage{amssymb}
\usepackage{mathrsfs}
\usepackage{color}
\usepackage{graphicx,bm}
\usepackage{float}
\usepackage{enumerate}
\usepackage{arydshln}  
\usepackage{pifont}  
\usepackage{amsthm,amscd}
\usepackage[all]{xy}

 \allowdisplaybreaks

\usepackage{hyperref}

\newtheorem{theorem}{Theorem}[section]

\newtheorem{definition}[theorem]{Definition}

\newtheorem{lemma}[theorem]{Lemma}
\newtheorem{claim}{Claim}[section]

\date{October 19, 2023}

\begin{document}

\title{Refinement on  spectral Tur\'{a}n's theorem\thanks{
This paper was firstly announced in April, 2022,
and was later published on SIAM J. Discrete Math. 37 (4) (2023) 2462--2485.  
See \url{https://doi.org/10.1137/22M1507814}. 
The research was supported by National Natural Science Foundation of China grant 11931002.
E-mail addresses: \url{ytli0921@hnu.edu.cn} (Y\v{o}ngt\={a}o L\v{i}),
\url{ypeng1@hnu.edu.cn} (Yu\`{e}ji\`{a}n P\'{e}ng, corresponding author).}}

\author{Yongtao Li, Yuejian Peng$^{\dag}$\\[2ex]
{\small School of Mathematics, Hunan University} \\
{\small Changsha, Hunan, 410082, P.R. China }  }

\maketitle

\vspace{-0.5cm}

\begin{abstract}
A well-known result in extremal spectral graph theory,
due to Nosal and Nikiforov, states that
if $G$ is a triangle-free graph on $n$ vertices, then $\lambda (G)
\le \lambda (K_{\lfloor \frac{n}{2}\rfloor, \lceil \frac{n}{2} \rceil })$, equality holds
if and only if $G=K_{\lfloor \frac{n}{2}\rfloor, \lceil \frac{n}{2} \rceil }$.
 Nikiforov [Linear Algebra Appl. 427 (2007)]
 extended this result to $K_{r+1}$-free graphs for every integer $r\ge 2$.
 This is known as the spectral Tur\'{a}n theorem.
 Recently, Lin, Ning and Wu [Combin. Probab. Comput. 30 (2021)]
proved a refinement on this result for non-bipartite triangle-free graphs.
In this paper, we provide alternative proofs for
the result of Nikiforov and the result of Lin, Ning and Wu.
Our  proof can allow us to extend
the later result to non-$r$-partite $K_{r+1}$-free graphs.
Our result  refines the theorem of Nikiforov and
it also can be viewed as a spectral version of a theorem of Brouwer.
 \end{abstract}

{{\bf Key words:}
Tur\'{a}n theorem;
Spectral radius;
Zykov symmetrization. }

{{\bf 2010 Mathematics Subject Classification.}  05C50, 05C35.}

\section{Introduction}

Extremal graph theory is becoming  one of the significant branches of
discrete mathematics nowadays, and it has experienced
an impressive growth during the last few decades.
With the rapid developments of combinatorial number theory and combinatorial geometry,
extremal graph theory has a large number of applications to these areas of mathematics.
Problems in extremal graph theory  deal usually with the question of determining or
estimating the maximum or minimum possible
size of graphs  satisfying certain requirements,
and further characterize the extremal graphs attaining the bound.
For example,
one of the most well-studied problems is the Tur\'{a}n-type problem,
which asks to determine the maximum number of edges
in a graph forbidding the occurence of some specific substructures.
Such problems are related to other areas including
theoretical computer science,
discrete geometry, information theory and number theory.

\subsection{The classical extremal graph problems}
Given a graph $F$,
we say that a graph $G$ is $F$-free if it does not contain
  an isomorphic copy of $F$ as a subgraph.
  For example, every bipartite graph is $C_{3}$-free, where $C_3$ is a triangle.
 The {\em Tur\'{a}n number} of a graph $F$, denoted by $\mathrm{ex}(n, F)$, is the maximum number of edges
  in an $F$-free $n$-vertex graph.
  An $F$-free graph on $n$ vertices  with $\mathrm{ex}(n, F)$ edges is called an {\em extremal graph} for $F$.
  We denote by $K_{s,t}$  the complete bipartite graph with parts of sizes $s$
  and $t$.
  Over a century old, a well-known theorem of Mantel \cite{Man1907} states that if 
  $G$ is  an $n$-vertex triangle-free graph, then
$e(G) \le
e(K_{\lfloor \frac{n}{2}\rfloor, \lceil \frac{n}{2} \rceil }) =  \lfloor {n^2}/{4} \rfloor ,$
equality holds if and only if  $G=K_{\lfloor \frac{n}{2}\rfloor, \lceil \frac{n}{2} \rceil }$.

    In 1941, Tur\'{a}n \cite{Turan41} studied the question of extending Mantel's theorem to $K_{r+1}$-free graphs.
 Let $T_r(n)$ denote the complete $r$-partite graph on $n$ vertices
 whose part sizes are as equal as possible. That is,
 $T_{r}(n)=K_{t_1,t_2,\ldots ,t_r}$ with $\sum_{i=1}^r t_i=n$
 and $|t_i-t_j| \le 1$ for $i\neq j$. Tur\'{a}n's theorem states that
if $G$ is  an $n$-vertex $K_{r+1}$-free,
then
$e(G)\le e(T_r(n)),$
equality holds if and only if $G$ is the $r$-partite Tur\'{a}n graph $T_r(n)$.

Many  different proofs of Tur\'{a}n's theorem could be found in the literature;
see \cite[pp. 269--273]{AZ2014} and \cite[pp. 294--301]{Bollobas78} for more details.
Furthermore,
there are various extensions and generalizations on Tur\'{a}n's theorem;
see, e.g., \cite{BT1981,Bon1983}.
Tur\'{a}n's theorem  implies
the numerical bound
\begin{equation}  \label{eq-weak-Turan}
e(G)\le \left(1-\frac{1}{r} \right) \frac{n^2}{2}
\end{equation}
for every $n$-vertex $K_{r+1}$-free graph $G$.
This  bound seems more concise and called the weak version of Tur\'{a}n's theorem.
The problem of determining $\mathrm{ex}(n, F)$
is usually referred to as the
Tur\'{a}n-type extremal graph problem.
  It is  a cornerstone of extremal graph theory
  to
  understand $\mathrm{ex}(n, F)$  for various graphs $F$;
  see \cite{FS13, Sim13} for comprehensive surveys.

\subsection{The spectral extremal graph problems}

Let $G$ be a simple graph on $n$ vertices.
The \emph{adjacency matrix} of $G$ is defined as
$A(G)=[a_{ij}]\in \mathbb{R}^{n\times n}$ where $a_{ij}=1$ if two vertices $v_i$ and $v_j$ are adjacent in $G$, and $a_{ij}=0$ otherwise.
We say that $G$ has eigenvalues $\lambda_1 , \lambda_2,\ldots ,\lambda_n$ if these values are eigenvalues of
the adjacency matrix $A(G)$.
Let $\lambda (G)$ be the maximum  value in absolute
 among all eigenvalues of $G$, which is
 known as the {\it spectral radius} of graph $G$.
The Perron--Frobenius theorem (see, e.g., \cite[p. 120--126]{Zhan13})  implies that
the spectral radius of a graph $G$ is actually
the largest eigenvalue of $G$ and it corresponds to a nonnegative eigenvector.  Moreover, if $G$ is  connected,
then  $A(G)$ is an irreducible nonnegative matrix,
$\lambda (G)$ is an eigenvalue with multiplicity one and there
exists an entry-wise positive eigenvector corresponding to $\lambda (G)$.

\medskip
The classical extremal graph problems
 usually study the maximum or minimum
number of edges that the extremal graphs can have.
Correspondingly,
the extremal spectral problems are well-studied in the literature.
In 1970, Nosal \cite{Nosal1970}
determined the largest spectral radius of a  triangle-free graph, which states that
if $G$ is a triangle-free graph with $m$ edges, then $\lambda (G)\le \sqrt{m}$.
In order to state this result accurately, we  borrow contributions from
Nikiforov's work \cite{Niki2009jctb}, which determined the extremal case of equality.
Thus we write it as in the following complete form.
 When we consider a graph with given number of edges,
we shall ignore the possible isolated vertices if there are no confusions.

 \begin{theorem}[Nosal, 1970] \label{thmnosal}
Let $G$ be a  graph  with $m$ edges.
If $G$ is triangle-free, then
\begin{equation}   \label{eq1}
\lambda (G)\le \sqrt{m} ,
\end{equation}
 equality holds if and only if
$G$ is a complete bipartite graph.
\end{theorem}

Theorem \ref{thmnosal} implies that if  $G$ is  bipartite, then
$  \lambda (G)\le \sqrt{m} $,
 equality holds if and only if
$G$ is a complete bipartite graph.
On the one hand, Theorem \ref{thmnosal} implies  Mantel's theorem.
Indeed, applying Rayleigh's inequality, we have
$\frac{2m}{n}\le \lambda (G)\le  \sqrt{m}$,
which yields $ m \le \lfloor {n^2}/{4} \rfloor$.
On the other hand,  applying Mantel's theorem to (\ref{eq1}), we obtain that
$ \lambda (G)\le \sqrt{m}  \le \sqrt{\lfloor {n^2}/{4}\rfloor} =\lambda (K_{\lfloor \frac{n}{2}\rfloor, \lceil \frac{n}{2} \rceil })$,
which is called the spectral Mantel theorem.

\medskip 
Over the past few years,
various extensions and generalizations on the Nosal--Nikiforov theorem have been obtained in the literature; see, e.g., \cite{Niki2002cpc,Niki2007laa2,Niki2009jctb,Wil1986} for extensions
to $K_{r+1}$-free graphs,
\cite{LNW2021,ZLS2021,ZS2022dm,LP2022oddcycle} for extensions of graphs with
given size.
In addition, many spectral extremal problems are also
obtained recently; see
\cite{CFTZ20,CDT2022} for the friendship graph and the odd wheel,
\cite{LP2021,DKLNTW2021} for intersecting odd cycles and
cliques, \cite{Wang2022} for a recent conjecture.
We recommend the surveys \cite{NikifSurvey,CZ2018,LFL2022} for interested readers.
The eigenvalues of the adjacency matrix sometimes can give some information
about the  structure of a graph.
There is a rich history on the study of bounding
the eigenvalues of a graph in terms of
various parameters;
see \cite{BN2007jctb} for spectral radius and cliques,
\cite{TT2017,LN2021} for eigenvalues of outerplanar and planar graphs.

In 1986,  Wilf  \cite{Wil1986}  provided the first result regarding
the spectral version of Tur\'{a}n's theorem and
proved that for every $n$-vertex $K_{r+1}$-free graph $G$, we have
\begin{equation}  \label{eqwilf}
 \lambda (G)\le \left(1-\frac{1}{r} \right)n.
\end{equation}
In 2002,  Nikiforov \cite{Niki2002cpc}
proved that  for every $m$-edge $K_{r+1}$-free graph $G$,
\begin{equation}  \label{eqniki}
 \lambda (G)\le \sqrt{2m\left( 1-\frac{1}{r} \right)}.
 \end{equation}
 The case of equality in (\ref{eqniki}) was later characterized in \cite{Niki2009jctb}.
Both
(\ref{eqwilf}) and (\ref{eqniki}) are direct consequences of
 Motzkin--Straus' theorem \cite{MS1965}.
Combining with $\frac{2m}{n} \le \lambda (G)$,
we  see that either (\ref{eqwilf}) or (\ref{eqniki})
 can imply (\ref{eq-weak-Turan}).
Moreover, using (\ref{eq-weak-Turan}), we  know that (\ref{eqniki})  implies (\ref{eqwilf}) immediately.

\medskip

In 2007, Nikiforov \cite{Niki2007laa2} showed a
spectral version of the Tur\'{a}n theorem.

\begin{theorem}[Nikiforov, 2007] \label{thm460}
Let $G$ be a  graph on $n$ vertices.
If $G$ is $K_{r+1}$-free, then
\[ \lambda (G)\le \lambda ({T_r(n)}),\]
equality holds if and only if  $G$
is the $r$-partite Tur\'{a}n graph $T_r(n)$.
\end{theorem}

Theorem \ref{thm460} implies Wilf's result (\ref{eqwilf}).
It should be mentioned that
the spectral version of  Tur\'{a}n's theorem was
also studied independently by Guiduli in his PH.D. dissertation \cite[pp. 58--61]{Gui1996}.
In 2021,
Lin, Ning and Wu \cite[Theorem 1.4]{LNW2021}
proved a generalization of Theorem \ref{thmnosal} for
non-bipartite triangle-free graphs (Theorem \ref{thmLNW}).
In this paper, we shall extend the result of Lin, Ning and Wu to
non-$r$-partite $K_{r+1}$-free graphs; see Theorem \ref{thm214}.
Our result is also a refinement on Theorem \ref{thm460}
in the sense of stability result.

\medskip
Assume that $T_1$, $T_2$, $\ldots ,T_r$ are vertex parts of Tur\'{a}n graph $T_r(n)$
with sizes $t_1,t_2,\ldots ,t_r$, respectively.
Moreover, we may assume further that $\lfloor \frac{n}{r}\rfloor = t_1\le t_2 \le \cdots \le t_r = \lceil \frac{n}{r}\rceil$.
Next, we are going to construct a new graph obtained from $T_r(n)$.

\begin{definition}[The extremal graph] \label{def-Yrn}
Choosing two parts $T_1$ and $T_r$ of the Tur\'{a}n graph $T_r(n)$,
 we add a new edge into the part $T_r$, denote by $uw$, and then remove all edges between $T_1$ and $\{u,w\}$.
Moreover, we connect $u$ to a vertex $v\in T_1$, and connect $w$ to the remaining
vertices of $T_1$.
The resulting graph is denoted by $Y_r(n)$;
see Figure  \ref{fig-5}.
\end{definition}

\begin{figure}[H]
\centering
\includegraphics[scale=0.75]{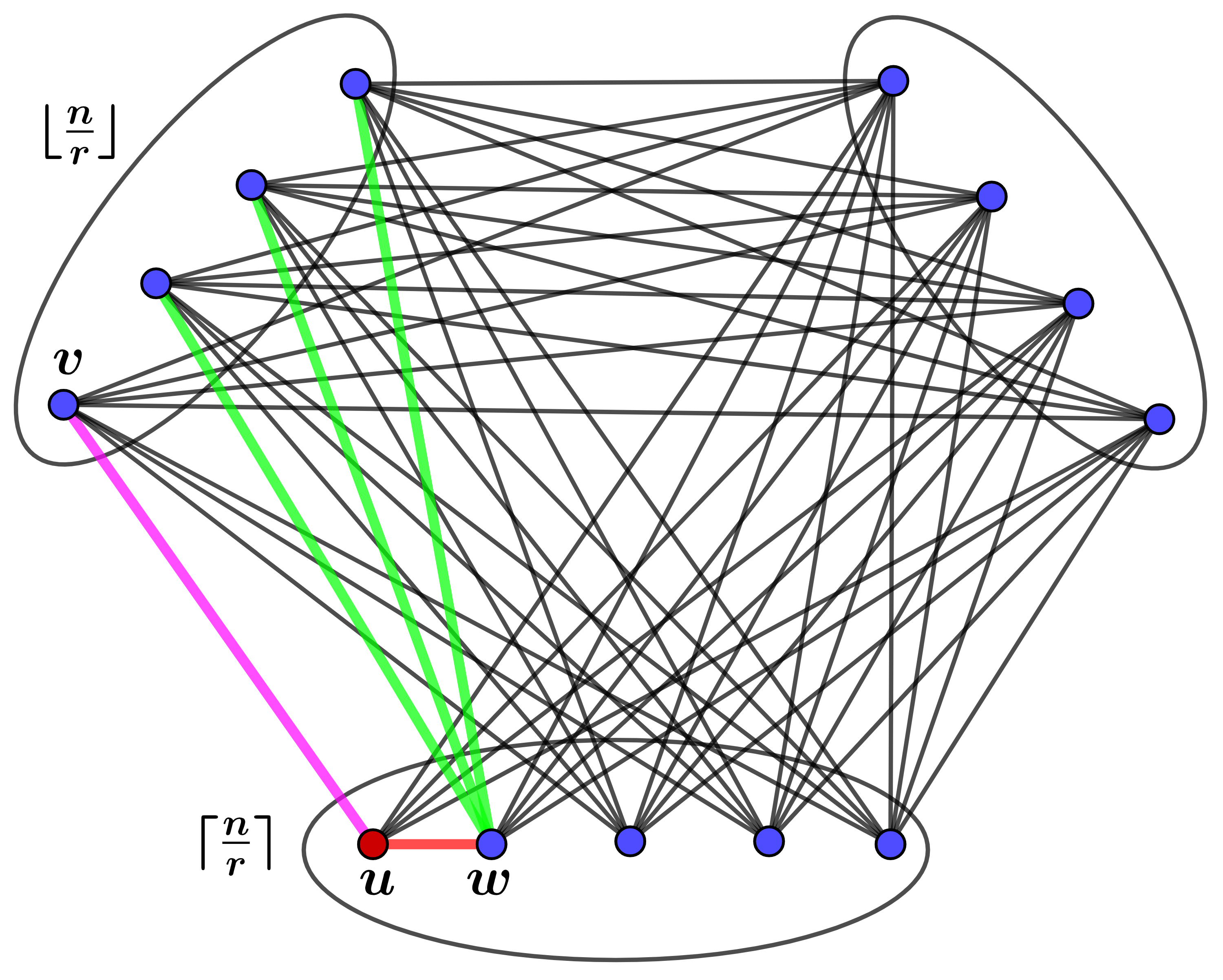}
\caption{The graph $Y_r(n)$ for $n=13 $ and $r=3$.}
    \label{fig-5}
\end{figure}

Now, we present the main result in this paper.

\begin{theorem}[Main result] \label{thm214}
 Let $G$ be an $n$-vertex non-$r$-partite $K_{r+1}$-free graph.
Then
\[  \lambda (G) \le \lambda (Y_r(n)).  \]
Moreover, the equality holds if and only if $G=Y_r(n)$.
\end{theorem}

This article is organized as follows.
In Section \ref{sec2}, we shall give an alternative proof of
the spectral Tur\'{a}n's theorem \ref{thm460}.
To make the proof of Theorem \ref{thm214} more transparent,
we will present a quite different proof of the triangle case of Lin, Ning and Wu \cite{LNW2021} in Section \ref{sec3}.
Inspired by the works \cite{Gui1996,HJZ2013,KN2014},
we shall use mainly the spectral Zykov symmetrization \cite{Zykov1949}.
In Section \ref{sec4},
we shall show the detailed proof of Theorem \ref{thm214}.
In Section \ref{sec5}, we shall discuss the spectral  extremal
problem in terms of the $p$-spectral radius.
Section \ref{sec6} contains some spectral problems for $F$-free graphs with
the chromatic number $\chi (G)\ge t$ and the problems in terms of
the signless Laplacian spectral radius.

\section{Alternative proof of Theorem \ref{thm460}}

\label{sec2}

The proof of Nikiforov \cite{Niki2007laa2} for Theorem \ref{thm460}  is more algebraic
and based on the characteristic polynomial of the complete $r$-partite graph.
Moreover, his proof relies on an inequality \cite{Niki2002cpc} relating the spectral radius and
the number of cliques,
as well as  an old theorem of Zykov \cite{Zykov1949} (see Erd\H{o}s \cite{Erd1962}), which asserts that
$k_s(G) \le k_s(T_r(n))$ for every $s\ge 2$, where $k_s(G)$ is the number of $s$-cliques in $G$.
This result is viewed as a clique extension of Tur\'{a}n's theorem.

The proof of Guiduli  \cite[pp. 58--61]{Gui1996}  for Theorem \ref{thm460}  is completely different from that of Nikiforov.
The main idea of  Guiduli's proof  reduces the problem for $K_{r+1}$-free graphs
to that for complete $r$-partite graphs by applying a spectral technique of Erd\H{o}s' degree majorization algorithm \cite{Erd1970}.
In this way, it is sufficient to show that
the Tur\'{a}n graph $T_r(n)$ attains the maximum spectral radius
among all complete $r$-partite graphs;
see, e.g., \cite{HJZ2013,KN2014} for more spectral applications,
and \cite{Fur2015,BBCLMS2017} for related topics.

\medskip 
In this section, we shall provide an alternative proof of
Theorem \ref{thm460}.
The proof is motivated by  the papers \cite{Gui1996,HJZ2013,KN2014},
and it is based on a spectral extension of the Zykov symmetrization \cite{Zykov1949},
which is becoming a powerful tool for extremal graph problems; see, e.g.,
\cite{FM2017} for a recent application on
the minimum number of triangular edges.

The following lemma was proved by Feng, Li and Zhang in \cite[Theorem 2.1]{FLZ2007}.

\begin{lemma}[Feng--Li--Zhang, 2007] \label{lem-FLZ}
If $G$ is an $r$-partite graph on $n$ vertices, then
\[  \lambda (G) \le \lambda (T_r(n)), \]
equality holds if and only if  $G$
is the $r$-partite Tur\'{a}n graph $T_r(n)$.
\end{lemma}

Now, we present our alternative proof of Theorem \ref{thm460}.

\begin{proof}[{\bf Proof of Theorem \ref{thm460}}]
Let $G$ be a $K_{r+1}$-free graph on $n$ vertices
with maximum value of the spectral radius.
Firstly, we  show that $G$ is a connected graph.
Otherwise, if $G$ is not connected, then adding a new edge between a component attaining the spectral radius
of $G$ and any other component will strictly increase the spectral radius of $G$,
and it does not create a copy of $K_{r+1}$.  Hence we get a new $K_{r+1}$-free graph with larger spectral radius,
which contradicts with the choice of $G$.
Since $G$ is connected,
we can take $\bm{x} \in \mathbb{R}^n$ as a positive unit eigenvector of $\lambda (G)$. Hence, we have
\[  \lambda (G)=  2
\sum_{\{i,j\} \in E(G)} x_ix_j . \]

Our goal is to show that $G$ is the Tur\'{a}n graph $T_r(n)$.
By Lemma \ref{lem-FLZ}, it suffices to show that
$G$ is a complete $r$-partite graph.
Suppose on the contrary that
$G$ is not complete $r$-partite.
Then there are three vertices $u,v,w \in V(G)$  such that
 $vu\notin E(G)$ and $uw\notin E(G)$, while $vw\in E(G)$.
 (This reveals that the non-edge relation
 between vertices is not an equivalent binary relation, as it does not satisfy the transitivity.)
 Throughout the paper,
 we denote by $s_G(v,\bm{x}) $ the sum of weights of
 vertices in $N_G(v)$. Namely,
 \[  \boxed{s_G(v,\bm{x}) :=\sum_{i\in N_G(v)} x_i.} \]

 \begin{figure}[htbp]
\centering
\includegraphics[scale=0.8]{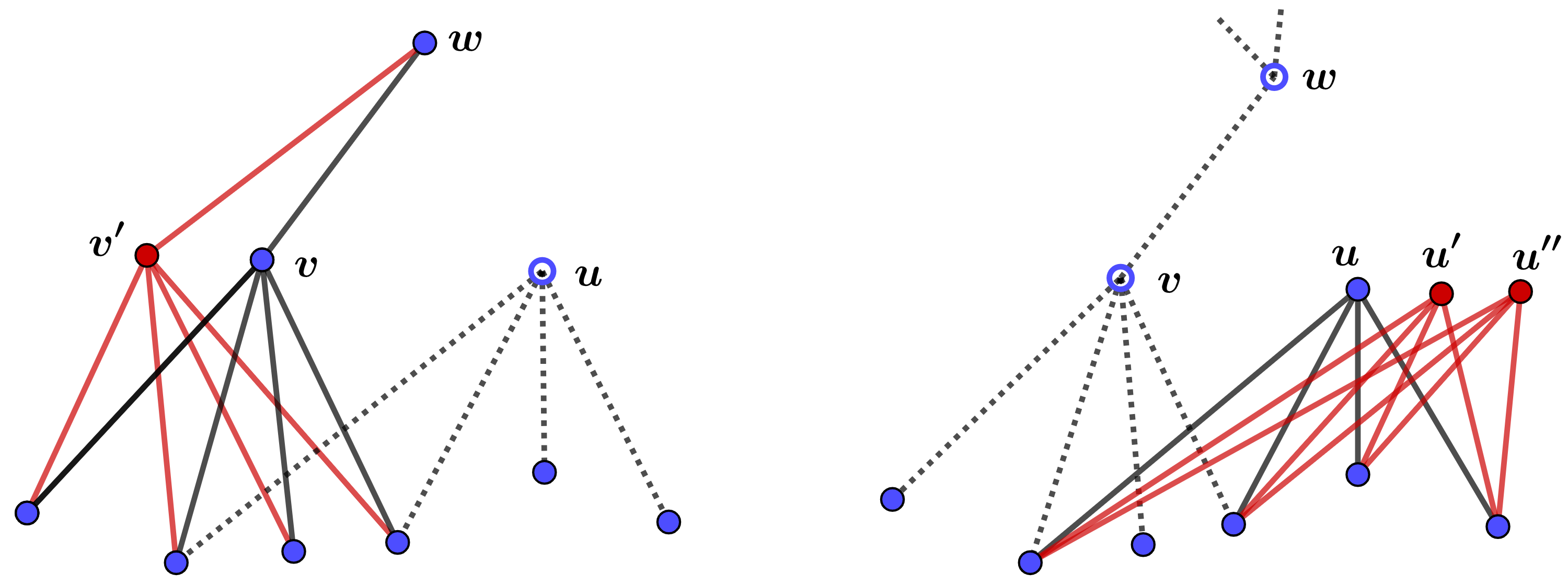}
\caption{The spectral Zykov symmetrization.} \label{fig-1}
\end{figure}

{\bf Case 1.}~ $s_G(u,\bm{x}) <  s_G(v,\bm{x}) $ or $ s_G(u,\bm{x})
< s_G(w,\bm{x}) $.

We may assume that $s_G(u,\bm{x}) <  s_G(v,\bm{x})$.
Then we duplicate the vertex $v$,
that is, we create a new vertex $v'$ which
has exactly the same neighbors as $v$, but $vv'$
is not an edge, and we delete the vertex $u$ and its incident edges;
see the left graph in Figure \ref{fig-1}.
Moreover, we distribute the value $x_u$ to the new vertex $v'$,
and keep the other coordinates of $\bm{x}$ unchanged.
It is not hard to verify that
 the new graph $G'$
 has still no copy of $K_{r+1}$ and
 \begin{equation*}
 \begin{aligned}
  \lambda (G')\ge
 2\sum_{\{i,j\} \in E(G')} x_ix_j
 &= 2 \sum_{\{i,j\} \in E(G)} x_ix_j -
 2x_u s_G(u,\bm{x})  + 2x_u s_G(v,\bm{x})\\
 &> 2 \sum_{\{i,j\} \in E(G)} x_ix_j
 =\lambda (G),
 \end{aligned}
 \end{equation*}
where we used the positivity of vector $\bm{x}$.
This contradicts with the choice of $G$.

{\bf Case 2.}~ $s_G(u,\bm{x}) \ge s_G(v,\bm{x})$ and
$s_G(u,\bm{x}) \ge s_G(w,\bm{x})$.

We copy the vertex $u$ twice, and delete both
$v$ and $w$
with their incident edges;
see the right graph in Figure \ref{fig-1}.
Similarly, we distribute the value $x_v$ to the new vertex $u'$,
and $x_w$ to the new vertex $u''$,
and keep the other coordinates of $\bm{x}$ unchanged.
Moreover, the new graph
$G''$ contains no copy of $K_{r+1}$ and
 \begin{equation*}
 \begin{aligned}
  \lambda (G'')\ge
2\sum_{\{i,j\} \in E(G'')} x_ix_j
 &=2
 \sum_{\{i,j\} \in E(G)} x_ix_j  - 2x_v s_G(v,\bm{x})  - 2x_w s_G(w,\bm{x})   \\
 & \quad \quad  \quad  \,\,\,\, + 2x_vx_w +
 2x_v s_G(u,\bm{x})+ 2x_w s_G(u,\bm{x}) \\
 &> \sum_{i=1}^n x_i s_G(i,\bm{x})
 =\lambda (G).
 \end{aligned}
 \end{equation*}
 So we get a contradiction again.
\end{proof}

We  conclude  that the spectral Zykov's symmetrization starts with a
$K_{r+1}$-free graph $G$, and at each step
takes two {\it non-adjacent} vertices $v_i$ and $v_j$
such that $s_G(v_i, \bm{x}) > s_G(v_j, \bm{x})$, and deleting all edges incident to $v_j$,
and adding new edges between vertex $v_j$ and the neighborhood $N(v_i)$.
We do the same if $s_G(v_i, \bm{x}) = s_G(v_j, \bm{x})$ and
$N(v_i) \neq N(v_j)$ for $i<j$.
The spectral  Zykov's symmetrization does not increase
the size of the largest clique and
does not decrease the spectral radius\footnote{Combining Rayleigh's formula or Lagrange's multiplier method,
one can show further that the spectral radius will increase strictly whenever all  coordinates of the vector $\bm{x}$ are positive.}.
When the process terminates, it yields a  complete multipartite graph with at most $r$ vertex parts.
Otherwise, there are three vertices $u,v,w\in V(G)$ such that
$vu\notin E(G)$ and $uw\notin E(G)$ but $vw\in E(G)$.
Applying the same argument in the proof of Theorem \ref{thm460},
we can get a new graph with larger spectral radius, a contradiction.

We illustrate the difference between
the spectral Erd\H{o}s degree majorization algorithm
and  the spectral Zykov symmetrization.
Recall that the spectral Erd\H{o}s degree majorization algorithm asks us to
choose a vertex $v\in V(G)$ with the maximum value of $s_G(v ,\bm{x})$
among all vertices of $G$, and
 remove all edges incident to vertices of $V(G)\setminus ( N_G(v) \cup \{v\})$, and then add all edges between $N_G(v)$ and $V(G)\setminus N_G(v)$.
This operation makes each vertex of $V(G)\setminus (N_G(v)\cup \{v\})$
 being a copy of the vertex $v$.
 Since $G$ is $K_{r+1}$-free, we see that
 the subgraph of $G$ induced by $N_G(v)$ is $K_r$-free.
 We denote by $V_1=V(G) \setminus N_G(v)$.
 Next, we do the same operation on vertex set ${V_1^c}=N_G(v)$.
More precisely, we further choose a vertex $u\in {V_1^c}$
with the maximum value of $s_G(u,\bm{x})$ over all vertices of $ V_1^c$,
and  remove all edges incident to vertices $V_1^c \setminus
(N_{V_1^c}(u) \cup \{u\})$, and then add all edges between
$N_{V_1^c}(u)$ and $V_1^c \setminus N_{V_1^c}(u)$.
Using this operation repeatedly,
we get a complete $r$-partite graph $H$ on the same vertex set $V(G)$.
Furthermore, one can verify that
the majorization inequality $s_G(v,\bm{x}) \le s_H(v, \bm{x})$ holds
for every vertex $v\in V(G)$; see, e.g., \cite{Gui1996, HJZ2013, KN2014}.

The spectral Erd\H{o}s  majorization algorithm
and  the spectral Zykov symmetrization share some similarities.
For example,  these two operations
ask us to compare the sum of weights of neighbors,
and turn a $K_{r+1}$-free graph to a complete $r$-partite graph.
Importantly, these two operations do not create a copy of $K_{r+1}$
and do not decrease the value of spectral radius.
The only difference between them is that one step of the Erd\H{o}s operation will change
  many vertices with its incident edges, while one step of the Zykov operation will  change only two vertices with its incident edges.
  This subtle difference will bring great convenience in later Sections \ref{sec3} and \ref{sec4}. As a matter of fact, at each step of the Erd\H{o}s operation,
there are many times of actions of the Zykov operation.
In other words, each step of the Erd\H{o}s operation
can be decomposed as a series of the Zykov operation.

\section{Refinement for triangle-free graphs}

\label{sec3}

Mantel's Theorem  has many interesting applications and
miscellaneous generalizations in the literature; see, e.g., \cite{Bollobas78,BT1981,Bon1983,Sim13} and references therein.
In particular,
Mantel's  theorem
was refined in the sense of the following stability form.

\begin{theorem}[Erd\H{o}s] \label{thmErd}
Let $G$ be an $n$-vertex triangle-free graph.
If $G$ is not bipartite, then
\[ e(G)\le \left\lfloor \frac{(n-1)^2}{4} \right\rfloor+1 .\]
\end{theorem}

\begin{figure}[H]
\centering
\includegraphics[scale=0.75]{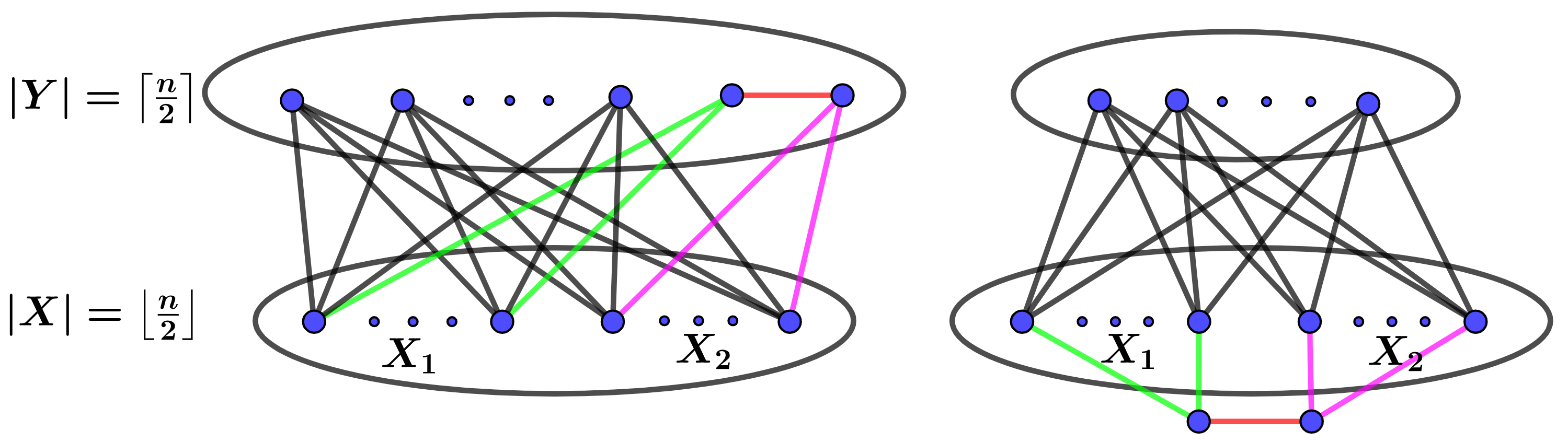}  \\
\caption{Two  drawings of extremal graphs in Theorem \ref{thmErd}.}
    \label{fig-2}
\end{figure}

It is said that this stability result attributes to Erd\H{o}s;
see \cite[Page 306]{BM2008}.
The bound in Theorem \ref{thmErd} is best possible and
the extremal graph is not unique.
Taking two vertex sets
$X$ and $Y$ with $|X|= \lfloor \frac{n}{2} \rfloor$
and $|Y|= \lceil \frac{n}{2} \rceil$,
we choose two  vertices $u,v \in Y$ and  join them,
 then we put every edge between $X$ and $Y\setminus \{u,v\}$.
Partitioning $X$ into two parts $X_1$ and $X_2$ {\it arbitrarily}
(this shows that the extremal graph is not unique),
 we connect $u$ to every vertex in $X_1$, and
$v$ to every vertex in $X_2$; see Figure \ref{fig-2}.
This yields a triangle-free graph $G$ and
 $e(G)= \lfloor \frac{n^2}{4}\rfloor - \lfloor \frac{n}{2}\rfloor +1
  = \lfloor \frac{(n-1)^2}{4} \rfloor +1 $. Note that $G$ has a $5$-cycle,
so it is not bipartite.

\medskip
In 2021, Lin, Ning and Wu \cite[Theorem 1.4]{LNW2021} proved
a  generalization on spectral Mantel theorem for non-bipartite graphs.
Let $SK_{\lfloor \frac{n-1}{2}\rfloor,\lceil \frac{n-1}{2}\rceil}$ denote the
subdivision of the complete bipartite graph
$K_{\lfloor \frac{n-1}{2}\rfloor,\lceil \frac{n-1}{2}\rceil}$ on one edge;
see  Figure \ref{fig-3}.
Clearly, $SK_{\lfloor \frac{n-1}{2}\rfloor,\lceil \frac{n-1}{2}\rceil}$
is one of the extremal graphs in Theorem \ref{thmErd}
by setting $|X_1|=\lfloor \frac{n}{2}\rfloor -1$ and $|X_2|=1$
in  Figure \ref{fig-2}.

\begin{figure}[htbp]
\centering
\includegraphics[scale=0.75]{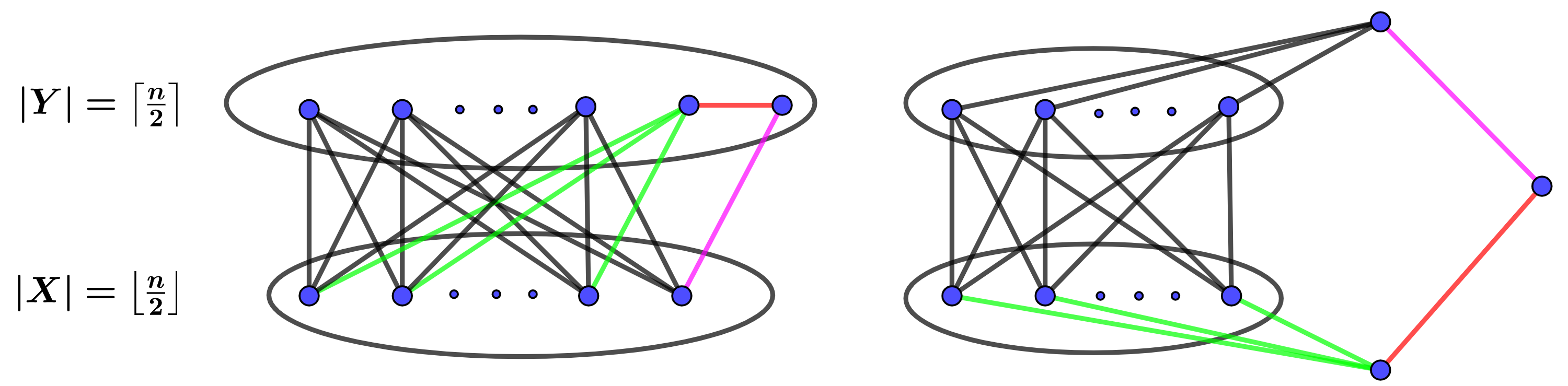}
\caption{Two  drawings of the graph $SK_{\lfloor \frac{n-1}{2}\rfloor,\lceil \frac{n-1}{2}\rceil}$.}
    \label{fig-3}
\end{figure}

\begin{theorem}[Lin--Ning--Wu, 2021] \label{thmLNW}
Let $G$ be an $n$-vertex graph.
If $G$ is triangle-free and non-bipartite,
then
\[  \lambda (G) \le \lambda (SK_{\lfloor \frac{n-1}{2}\rfloor,\lceil \frac{n-1}{2}\rceil}),  \]
equality holds if and only if $G=SK_{\lfloor \frac{n-1}{2}\rfloor,\lceil \frac{n-1}{2}\rceil}$.
\end{theorem}

Theorem \ref{thmLNW}
is  a corresponding spectral version of Theorem \ref{thmErd},
while the extremal graph in spectral problem is unique;
see \cite{LG2021,LSY2022} for an extension to graphs without short odd cycles.
In this section,
we shall provide a new proof of Theorem \ref{thmLNW}.
One of the key ideas in the proof is to use the spectral Zykov symmetrization,
which provides great convenience
to obtain a clearly approximate structure of the required extremal graph.
Moreover, the ideas  in this proof can benefit us to extend
Theorem \ref{thmLNW} to  $K_{r+1}$-free non-$r$-partite graphs,
which will be discussed in Section \ref{sec4}.
Before starting the proof, we include the following lemma,
which is a direct consequence by computations; see, e.g., \cite[Appendix A]{LNW2021}.

\begin{lemma} \label{lempp}
If $G$ is a graph on $n=a+b+1$ vertices
obtained
from $K_{a,b}$ by subdividing an edge arbitrarily, then
\[  \lambda (G) \le \lambda (SK_{\lfloor \frac{n-1}{2}\rfloor,\lceil \frac{n-1}{2}\rceil}),  \]
equality holds if and only if $G=SK_{\lfloor \frac{n-1}{2}\rfloor,\lceil \frac{n-1}{2}\rceil}$.
\end{lemma}

\begin{proof}
We denote by $SK_{a,b}$ the graph obtained from $K_{a,b}$ by subdividing an edge.
Let $s,t$ be two positive integers with $t\ge s\ge 1$.
It suffices to show that
\[ \lambda (SK_{s+1,t+3})  < \lambda (SK_{s+2,t+2}).  \]
By computation, the spectral radius of $SK_{a,b}$ is the largest root of
\begin{equation*}
\begin{aligned}
F_{a,b}(x) :=x^5 - (ab+1) x^3+ (3ab -2a -2b +1)x -2ab +2a +2b -2.
 \end{aligned}
 \end{equation*}
Hence $\lambda(SK_{s+2,t+2})$ is the largest root of
\[  F_{s+2,t+2}(x)= x^5 - (2s+2t +st +5)x^3 +
(4s+4t +3st +5)x -2s-2t -2st -2. \]
Similarly, $\lambda (SK_{s+1,t+3})$ is the largest root of $F_{s+1,t+3}(x)$.
Note that
\[ F_{s+2,t+2}(x) - F_{s+1,t+3}(x) = - (x-1)^2(x+2)(t-s+1). \]
This implies $  F_{s+2,t+2}(x) < F_{s+1,t+3}(x)$ for every $x>1$.
Since $K_{2,3}$ is a subgraph of $SK_{s+1,t+3}$, we know that
$\lambda (SK_{s+1,t+3}) \ge \lambda (K_{2,3})=\sqrt{6}$. Thus, we have
\[  F_{s+2,t+2}(\lambda (SK_{s+1,t+3})) <
F_{s+1,t+3}(\lambda (SK_{s+1,t+3}))=0. \]
Therefore, we obtain  $\lambda (SK_{s+1,t+3}) < \lambda (SK_{s+2,t+2})$.
\end{proof}

Now we are ready to show our proof of Theorem \ref{thmLNW}.
For two {\it non-adjacent} vertices $u,v\in V(G)$,
we denote the  Zykov symmetrization $Z_{u,v}(G)$
to be the graph obtained from $G$ by replacing $u$ with a twin of $v$, that is,
deleting all edges incident to vertex $u$, and then adding new edges from $u$ to $N_G(v)$.
We can verify that the  Zykov symmetrization does not
increase both the clique number  $\omega (G)$ and
the chromatic number $\chi (G)$.
More precisely,
we have $\omega (Z_{u,v}(G)) = \omega (G \setminus \{u\})$ and
 $\chi (Z_{u,v}(G)) = \chi (G \setminus \{u\})$.
 Let $\bm{x} \in \mathbb{R}^n$ be a {\it positive unit eigenvector}  corresponding to $\lambda (G)$.
Recall that $s_G(v,\bm{x}) :=\sum_{i\in N_G(v)} x_i$ denotes
the sum of weights of
all neighbors of $v$ in $G$.

If $s_G(u,\bm{x}) < s_G(v,\bm{x})$,
then we replace $G$ with $Z_{u,v}(G)$. Apparently,
the spectral Zykov symmetrization does not make triangles.
More importantly, it will increase strictly the spectral radius, since
\begin{equation*}
 \begin{aligned}
\lambda (Z_{u,v}(G)) \ge
 2\sum_{\{i,j\} \in E(Z_{u,v}(G))} x_ix_j
 &= 2 \sum_{\{i,j\} \in E(G)} x_ix_j -
 2x_u s_G(u,\bm{x})  + 2x_u s_G(v,\bm{x})\\
 &> 2 \sum_{\{i,j\} \in E(G)} x_ix_j
 =\lambda  (G).
 \end{aligned}
 \end{equation*}

If $s_G(u,\bm{x}) = s_G(v,\bm{x})$
and $N_G(u)\neq N_G(v)$, then we can apply either $Z_{u,v}$ or $Z_{v,u}$.
In each case, we will get a new graph such that  $N(u)=N(v)$.
Similarly, this operation will increase the spectral radius $\lambda (G)$ strictly.
Indeed, we can  see that
\[  \lambda (Z_{u,v}(G)) \ge
 2\sum_{\{i,j\} \in E(Z_{u,v}(G))} x_ix_j
 = 2 \sum_{\{i,j\} \in E(G)} x_ix_j
 =\lambda  (G). \]
  We claim further that $\lambda (Z_{u,v}(G)) > \lambda (G)$.
 Assume on the contrary that $\lambda (Z_{u,v}(G)) = \lambda (G)$,
then the inequality in above becomes an equality.
 Thus $\bm{x}$ is an eigenvector of $\lambda (Z_{u,v}(G))$,
 namely,  $A(Z_{u,v}(G))\bm{x} = \lambda (Z_{u,v}(G)) \bm{x}
 =\lambda (G)\bm{x}$. Taking any vertex $z\in N_G(v) \setminus N_G(u)$,
 we observe that
 \[   \lambda (Z_{u,v}(G))x_z=\sum_{t\in N_G(z)\cup \{u\}} x_t
 > \sum_{t\in N_G(z)} x_t = \lambda (G)x_z.\]
 Consequently, we get $\lambda (Z_{u,v}(G)) > \lambda (G)$, which
 contradicts with our assumption.
It is worth emphasizing that the positivity of $\bm{x}$ is necessary in above discussions.
Roughly speaking, applying the spectral Zykov symmetrization
will make a $K_{r+1}$-free graph more regular in some sense
according to the weights of the eigenvector.

\begin{proof}[{\bf Proof of Theorem \ref{thmLNW}}]
Let $G$ be a non-bipartite triangle-free graph on $n$ vertices with the largest
spectral radius. Our goal is to show that
$G=SK_{\lfloor \frac{n-1}{2}\rfloor,\lceil \frac{n-1}{2}\rceil}$.
Clearly, we know that $G$ is connected. Otherwise, any addition of
an edge between a component with the maximum spectral radius
and any other component  will strictly increase the spectral radius.
Since $G$ is connected,
there exists  a positive unit eigenvector corresponding to $\lambda (G)$, and then we denote
such a vector by $\bm{x}=(x_1,\ldots ,x_n)^T$,
where $x_i>0$ for every $i$.
Since $G$ is triangle-free,
we  apply repeatedly the spectral Zykov symmetrization for every pair of non-adjacent vertices
until it becomes a bipartite graph. Without  loss of generality,
we may assume that $G$ is triangle-free and non-bipartite, while
$Z_{u,v}(G)$ is bipartite. We  are going to show that $\lambda (G) \le
\lambda (SK_{\lfloor \frac{n-1}{2}\rfloor,\lceil \frac{n-1}{2}\rceil})$,
equality holds if and only if $G=SK_{\lfloor \frac{n-1}{2}\rfloor,\lceil \frac{n-1}{2}\rceil}$.

Since $Z_{u,v}(G)$ is bipartite, we know that $G\setminus \{u\}$ is
bipartite. We denote $V(G)\setminus \{u\}=V_1\cup V_2 $,
where $V_1,V_2$ are disjoint and $|V_1| + |V_2|=n-1$.
Assume that $C=N(u)\cap V_1$ and $D=N(u)\cap V_2$. We denote
$A=V_1 \setminus C$ and $B=V_2\setminus D$.
Since $G$ is triangle-free, there are no edges between parts $C$ and $D$.
As $G$ attains the largest spectral radius, we know that
the pair of parts $(A,B),(A,D)$ and $(B,C)$ are complete bipartite subgraphs;
see Figure \ref{fig-4}.

\begin{figure}[htbp]
\centering
\includegraphics[scale=0.8]{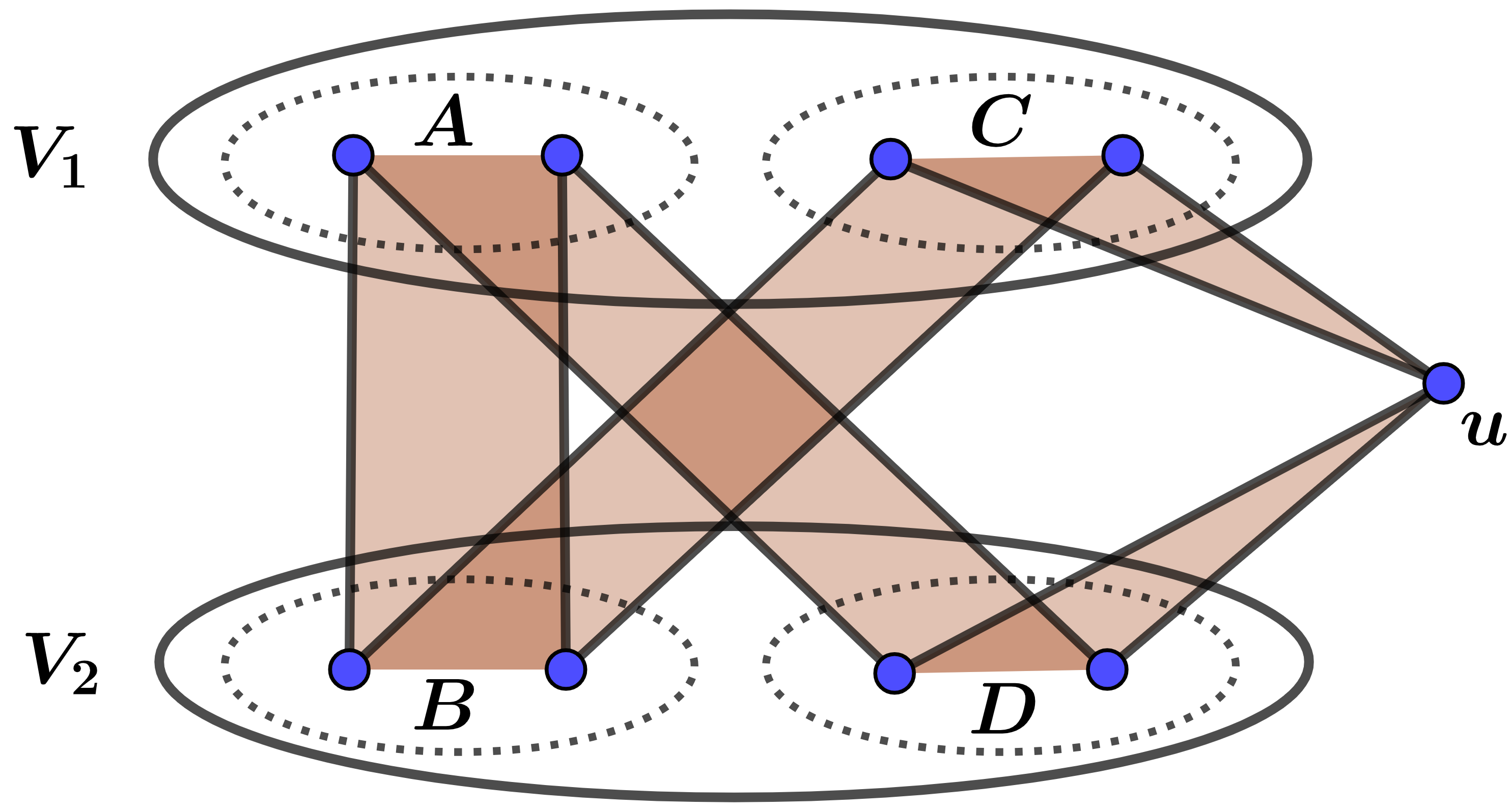}
\caption{An approximate structure of $G$.}
    \label{fig-4}
\end{figure}

Note that each vertex in  $A$ has the same neighborhood,
we know that the coordinates  $\{x_v:v\in A\}$ are all equal.
This property holds similarly for vertices in $B,C$ and $D$ respectively.
Thus, we write $x_a$ for the value of the entries of $\bm{x}$ in vertex set $A$.
And $x_b,x_c $ and $x_d$ are defined similarly.

The remaining steps of our proof are outlined as follows.

\ding{73}~
If $|A|x_a \ge |B|x_b$, then we delete $|C|-1$ vertices in $C$ with its incident edges,
and add $|C|-1$ new vertices to $D$ and connect these vertices to $A\cup \{u\}$.
We keep the weight of these new vertices being $x_c$ and
 denote the new graph by $G'$. We can verify that
\begin{align*}
\lambda (G') &\ge 2 \sum_{\{i,j\}\in E(G)} x_ix_j - 2(|C|-1)|B| x_cx_b
+ 2(|C|-1)|A|x_cx_a  \\
& \ge 2\sum_{\{i,j\}\in E(G)} x_ix_j =\lambda (G).
\end{align*}
In fact, we can further prove that $\lambda (G') > \lambda (G)$.
Otherwise, if $\lambda (G')= \lambda (G)$,
then $\bm{x}$ is the Perron vector of $G'$, that is,
$A(G')=\lambda (G')\bm{x}=\lambda (G)\bm{x}$.
Taking any vertex $z\in A$, we observe that
$\lambda (G')x_z=\sum_{v\in N_{G'}(z)} x_v =
\sum_{v\in N_G(z)} x_v + (|C|-1)x_c > \sum_{v\in N_G(z)} x_v
= \lambda (G) x_z$, and then $\lambda (G') > \lambda (G)$,
which is a contradiction.

\ding{78}~
If $|A|x_a < |B|x_b$, then we can delete $|D|-1$ vertices from $D$ with its incident edges,
and add $|D|-1$ new vertices to $C$ and join these new vertices to every vertex of $B\cup \{u\}$.
Similarly, we can show that this process will increase  the spectral radius strictly.
From the above discussion, we can always remove the vertices to force
either $|C|=1$ or $|D|=1$.
Without loss of generality, we may assume  that $|C|=1$ and $C=\{c\}$.

\ding{77}~
If $x_u\ge x_c$, then we remove $|B|-1$ vertices from $B$ with its incident edges,
and add $|B|-1$ new vertices to $D$ and join these vertices to $A\cup \{u\}$.
We keep the weight of these new vertices being $x_b$ and denote the new graph by
$G^*$. Then
\begin{align*}
\lambda (G^*) &\ge
2 \sum_{\{i,j\}\in E(G)} x_ix_j - 2(|B|-1)x_bx_c +2(|B|-1)x_bx_u  \\
& \ge
2 \sum_{\{i,j\}\in E(G)} x_ix_j  =\lambda (G).
\end{align*}
Furthermore, by Rayleigh's formula, we know that the first inequality holds strictly.
Thus we conclude  in the new graph $G^*$ that $B$ is a single vertex, say $B=\{b\}$.
We observe that the graph $G^*$ is a subdivision of a complete bipartite
graph on $(A\cup \{u\}, \{b\} \cup D)$ by subdividing the edge $\{b,u\}$.

\ding{72}~
If $x_u< x_c$, then we delete $|D|-1$ vertices from $D$ with its incident edges,
and add $|D|-1$ new vertices to $B$ and join these new vertices to $A\cup \{c\}$.
Keeping the weight of vertices unchanged, we denote the new graph by $G^{\star}$. Then
we can similarly get $\lambda (G^{\star}) > \lambda (G)$.
In the graph $G^{\star}$, we have $|D|=1$ and write $D=\{d\}$. Thus $G^{\star}$
is a subdivision of a complete bipartite
graph on $(A\cup \{c\}, B\cup \{d\})$ by subdividing the edge $\{c,d\}$.

From our discussion above, we know that if $G$ is an $n$-vertex triangle-free non-bipartite graph and
attains the maximum spectral
radius, then $G$ is a subdivision of a complete bipartite by subdividing exactly one edge. Lemma \ref{lempp}
implies that $G$ is a subdivision of a balanced complete bipartite graph
on $n-1$ vertices.
\end{proof}

\section{Refinement of spectral Tur\'{a}n theorem}

 \label{sec4}

In 1981, Brouwer \cite{Bro1981} proved the following improvement on
Tur\'{a}n's Theorem. 

\begin{theorem}[Brouwer, 1981]  \label{thmBrouwer}
Let  $n\ge 2r+1$ be an integer and
$G$ be an $n$-vertex graph.
If $G$ is $K_{r+1}$-free and $G$ is not $r$-partite,  then
  \[   e(G)\le e(T_r(n)) - \left\lfloor \frac{n}{r} \right\rfloor +1.  \]
\end{theorem}

Theorem \ref{thmBrouwer} was also independently studied in
many references, e.g., \cite{AFGS2013,HT1991,KP2005,TU2015}.
Similar with that of Theorem \ref{thmErd}, the bound of Theorem \ref{thmBrouwer} is sharp and there are many extremal graphs attaining this bound.  In particular, the graph $Y_r(n)$
in Definition \ref{def-Yrn}
is one of the extremal graphs of Brouwer's theorem.

We would like to illustrate the reason why we are interested in the study of
the family of non-$r$-partite graphs.
On the one hand, the Erd\H{o}s degree majorization algorithm \cite{Erd1970}
or \cite[pp. 295--296]{Bollobas78}
implies that if $G$ is an $n$-vertex $K_{r+1}$-free graph, then there exists an $r$-partite graph $H$
on the same vertex set $V(G)$ such that $d_G(v) \le d_H(v)$ for every vertex $v$.
Consequently, we get $e(G)\le e(H)\le e(T_r(n))$. Hence it is meaningful to
determine the family of graphs attaining the second largest value of
the extremal function. This problem is usually called the stability problem.
On the other hand,
there are various ways to study the extremal graph problems
under  some reasonable constraints.
For example, the condition of non-$r$-partite graph
is equivalent to saying the chromatic number $\chi (G)\ge r+1$.
Moreover, one can also consider the extremal problem under the
restriction $\alpha (G) \le f(n)$ for a given function $f(n)$,
where $\alpha (G)$ is the independence number of $G$.
This is the well-known Ramsey--Tur\'{a}n problem; see \cite{SS2001}
for a comprehensive survey.

The proof of Theorem \ref{thmLNW} stated in Section \ref{sec3}
can bring us
more effective treatment for the extremal spectral problem
when $K_{r+1}$  is a forbidden subgraph.
As promised in Introduction,
we shall prove Theorem \ref{thm214}, which
 extends Theorem \ref{thmLNW} to non-$r$-partite
$K_{r+1}$-free graphs.
Next, we restate Theorem \ref{thm214} as below for convenience
of readers.

\begin{theorem} \label{thm214-restate}
 Let $G$ be an $n$-vertex $K_{r+1}$-free graph.
If  $G$ is not $r$-partite,
then
\[  \lambda (G) \le \lambda (Y_r(n)).  \]
Moreover, the equality holds if and only if $G=Y_r(n)$.
\end{theorem}

Theorem \ref{thm214-restate} is not only
a spectral version of  Theorem \ref{thmBrouwer}, but also
a refinement of the spectral Tur\'{a}n's theorem \ref{thm460}.
  Our proof is mainly based on the spectral Zykov symmetrization.
Before showing the proof, we need to introduce the following lemma.

\begin{lemma} \label{lem42}
Let $K_{b_1,b_2,\ldots ,b_r}$ be the complete $r$-partite graph
with parts $B_1,B_2,\ldots ,B_r$ satisfying $|B_i|=b_i$ for every $i\in [r]$ and
$\sum_{i=1}^r b_i=n-1$.
Let $G$ be an $n$-vertex graph obtained from $K_{b_1,b_2,\ldots ,b_r}$
by adding a new vertex $u$ and choosing $v\in B_1, w\in B_2$,
and removing the edge $vw$, and adding the edges $uv, uw$ and $ut$ for every $t\in \cup_{i=3}^r B_i$. Then
\[  \lambda (G) \le \lambda (Y_r(n)). \]
Moreover, the equality holds if and only if $G=Y_r(n)$.
\end{lemma}

We illustrate the construction of $Y_r(n)$ in another way.
Let $T_r(n-1)$ be the $r$-partite Tur\'{a}n graph on $n-1$ vertices whose
parts $S_1$, $S_2, \ldots $, $S_r$ have sizes $s_1,s_2,\ldots ,s_r$ such that
$\lfloor \frac{n-1}{r}\rfloor = s_1\le s_2 \le \cdots \le s_r = \lceil \frac{n-1}{r}\rceil$.
Note that the extremal graph $Y_r(n)$ could be obtained from $T_r(n-1)$
by adding a new vertex $u$, and choosing two vertices $v\in S_1$ and $w\in S_2$,
then deleting the edge $vw$, and adding the edges $uv,uw$ and $ut$ for every
vertex $t \in \cup_{i=3}^r S_i$.
 Lemma \ref{lem42} states that $G$ attains the maximum spectral radius
only when its part sizes $b_1,b_2,\ldots ,b_r$ are as equal as possible,
and the two special vertices $v,w$ are located in the smallest two parts, respectively.
Since $\lambda (G)$ is the largest root of the
characteristic polynomial $P_G(x)=\det (xI_n- A(G))$,
it is operable to compute  $\lambda (G)$ exactly for some small integers $r$
by using computers, while it seems complicated for large $r$.

\begin{proof}[{\bf Proof of Lemma \ref{lem42}}]
Let $G$ be a graph satisfying the requirement of Lemma \ref{lem42}
and $G$ has the maximum spectral radius. We will show that $G=Y_r(n)$.
Since $G$ is connected, there exists a
positive unit eigenvector $\bm{x}\in \mathbb{R}^n$ corresponding to $\lambda (G)$.
Then $A(G) \bm{x}=\lambda (G) \bm{x}$ and
\[  \lambda (G)=
\bm{x}^TA(G)\bm{x}=2\sum_{\{i,j\}\in E(G)} x_ix_j. \]
Moreover, the eigen-equation gives that
$ \lambda (G) x_v= \sum_{u\in N(v)} x_u$ for every $v\in V(G)$.
It follows that  if  two non-adjacent vertices
have the same neighborhood, then they
have the same value on the corresponding coordinates of $\bm{x}$.
Thus all coordinates of $\bm{x}$ corresponding to the vertices of  $B_i$ are equal, and then we write $x_i$
for the value of those coordinates for each $i\in \{3,\ldots ,r\}$.
We denote $B_1^-=B_1 \setminus \{v\}$
and $B_2^-=B_2 \setminus \{w\}$. Similarly, all coordinates of $\bm{x}$
corresponding to the vertices of $B_i^-$ are equal for $i\in \{1,2\}$.

Assume on the contrary that $G$ is not isomorphic to $Y_r(n)$.
In other words,  there are two parts $B_i$ and $B_j$ such that $|b_i-b_j| \ge 2$, or $b_i \le b_j-1$ for some $i\in \{3,4,\ldots ,n\}$
and $j\in \{1,2\}$.
By the symmetry, there are  four cases listed below.
\begin{itemize}
\item[(A)] $b_i \le b_j-2$ for some $i,j\in \{3,\ldots ,r\}$;

\item[(B)] $b_1 \le b_2 -2$;

\item[(C)] $b_1\le b_i-2$ for some $i\in \{3,\ldots ,r\}$;

\item[(D)] $b_i\le b_1-1 $ for some $i\in \{3,\ldots ,r\}$.
\end{itemize}

{\bf Case A.}
First and foremost, we shall consider  case  that $b_i \le b_j -2$ for some
$i, j \in \{3,\ldots , r\}$. The treatment for this case has its root in \cite{KNY2015}.
If $b_i+b_j=2b$ for some integer $b$,
then we will balance the number of vertices of parts $B_i$ and $B_j$.
Namely, we define a new graph $G'$ obtained from $G$ by
deleting all edges between $B_i$ and $B_j$, and then we move some vertices
from $B_j$ to $B_i$ such that the resulting sets, say $B_i',B_j'$,
have size $b$, and then we add all edges between $B_i'$ and $B_j'$.
In this process, we keep the other edges unchanged.
We define a new vector $\bm{y}\in \mathbb{R}^n$ by setting
$y_s=\sqrt{(b_ix_i^2 +b_jx_j^2)/(2b)}$ for each vertex $s \in B_i'\cup B_j'$,
and $y_t=x_t$ for each $t \in V(G')\setminus (B_i' \cup B_j')$.
Then $\sum_{v\in V(G')} y_v^2=1$ and
\begin{align*}
\bm{y}^TA(G')\bm{y} - \bm{x}^TA(G)\bm{x}
 = 2((b y_s)^2- b_ix_ib_jx_j)
+ 2(2b y_s- (b_ix_i + b_j x_j))\sum_{t\notin B_i'\cup B_j'} x_t.
 \end{align*}
 Note that $b=\frac{b_i+b_j}{2}> \sqrt{b_ib_j}$ and
 \[ (b y_s)^2=b^2\cdot  \frac{b_ix_i^2 +b_jx_j^2}{2b}
 \ge b \sqrt{b_ix_i^2b_jx_j^2}> b_ix_ib_jx_j. \]
 Moreover, the weighted power-mean inequality gives
 \[ 2by_s =2b \left(\frac{b_ix_i^2 +b_jx_j^2}{b_i+b_j} \right)^{1/2}
 \ge 2b \frac{b_ix_i +b_jx_j}{b_i+b_j}=b_ix_i +b_jx_j. \]
 Thus we get $\bm{y}^TA(G')\bm{y} > \bm{x}^TA(G)\bm{x}$.
Rayleigh's formula gives
\[  \lambda (G') \ge \bm{y}^TA(G')\bm{y} >
\bm{x}^TA(G)\bm{x} =  \lambda (G), \]
 which contradicts with the choice of $G$.

If $b_i+b_j=2b+1$  for some integer $b$, then we move similarly
some vertices from $B_j$ to $B_i$ such that the resulting sets $B_i',B_j'$
satisfying $|B_i'|=b$ and $|B_j'|=b+1$. We construct a vector
$\bm{y}\in \mathbb{R}^n$ by setting
$y_s=\sqrt{(b_ix_i^2 +b_jx_j^2)/(2b+1)}$ for every vertex $s \in B_i'\cup B_j'$,
and $y_t=x_t$ for every $t \in V(G')\setminus (B_i' \cup B_j')$. Similarly, we get
\begin{equation*}
\begin{aligned}
\bm{y}^TA(G')\bm{y} - \bm{x}^TA(G)\bm{x}
 &= 2( b(b+1) y_s^2- b_ix_ib_jx_j)  \\
& \quad + 2( (2b+1) y_s- (b_ix_i + b_j x_j))\sum_{t\notin B_i'\cup B_j'} x_t.
 \end{aligned}
 \end{equation*}
We are going to show that
\[ b(b+1) y_s^2- b_ix_ib_jx_j> 0, \quad \text{and} \quad
(2b+1) y_s- (b_ix_i + b_j x_j) \ge 0. \]
For the first inequality, by applying AM-GM inequality, we get
\[ b(b+1)y_s^2 =b(b+1)\frac{b_ix_i^2+b_jx_j^2}{b_i+b_j}
\ge \frac{2b(b+1)}{b_i+b_j} \sqrt{b_ib_j} x_ix_j. \]
 It is sufficient to prove that
$2b(b+1) > (b_i+b_j)\sqrt{b_ib_j}$. Note that $b_i\le b_j-2$ and $b_i+b_j =2b+1$
is odd. Then $b_i \le b-1$ and $ b_j\ge b+2$. Thus, the first desired inequality holds immediately.
For the second one, the weighted power-mean inequality yields
 \[ (2b+1)y_s =(2b+1) \left(\frac{b_ix_i^2 +b_jx_j^2}{b_i+b_j} \right)^{1/2}
 \ge (2b+1) \frac{b_ix_i +b_jx_j}{b_i+b_j}=b_ix_i +b_jx_j. \]
This case also contradicts with the choice of $G$.

For the remaining three cases, we will show our proof
by considering the characteristic polynomial of the graph $G$
and then applying induction on integer $r$.

{\bf Case B.}
Now, we  consider the case $b_1\le b_2-2$.
Recall that $B_1^-=B_1\setminus \{v\}$
and  $B_2^-=B_2\setminus \{w\}$.
We define a graph $G'$ obtained from $G$  by deleting a
vertex of $B_2^-$, and adding a copy of a vertex of $B_1^-$.
This makes the two parts $B_1^-,B_2^-$ more balanced.
Our goal is to prove that $\lambda (G) < \lambda (G')$,
which contradicts with the maximality of $G$.
Let $x_v, x_w$ and $x_u$ be the weights of vertices
$v,w$ and $u$ respectively.
We denote by $x_1^-$ and $x_2^-$ the weights of vertices of
$B_1^-$ and $B_2^-$ respectively.
The eigen-equation $A(G)\bm{x}=\lambda (G) \bm{x}$ gives
$\sum_{j\in N(i)} x_j = \lambda (G) x_i$
for every $i\in [n]$. Then
\begin{equation*}
\begin{cases}
 \phantom{x_v+x_w+}  x_u
\phantom{+(b_1-1)x_1^-} + (b_2-1)x_2^- + b_3x_3 + \cdots + b_rx_r
=\lambda (G) x_v, \\
\phantom{x_v+x_w+}
x_u + (b_1-1)x_1^- \phantom{+(b_2-1)x_2^-}
+ b_3x_3 + \cdots +b_r x_r = \lambda (G) x_w, \\
 x_v +x_w
 \phantom{+x_u+(b_1-1)x_1^-+(b_2-1)x_2^-}
 \!\! + b_3x_3 + \cdots +b_rx_r = \lambda (G) x_u, \\
 \phantom{x_v+}\, \,x_w \phantom{+x_u +(b_1-1)x_1^-} + (b_2-1)x_2^- \! + b_3x_3 + \cdots +b_rx_r = \lambda (G) x_1^-, \\
 x_v \phantom{+x_w+x_u} \, + (b_1-1)x_1^-
\phantom{+(b_2-1)x_2^-} + b_3x_3 + \cdots +b_rx_r
=\lambda (G) x_2^- , \\
x_v+x_w+x_u + (b_1\!-\!1)x_1^-
+ (b_2\!-\!1)x_2^- + b_4x_4 + \cdots +b_rx_r = \lambda (G) x_3, \\
 \quad \,\,\,  \vdots \\
x_v+x_w+x_u + (b_1\!-\!1)x_1^- + (b_2\!-\!1)x_2^- + b_3x_3 + \cdots +b_{r-1}x_{r-1} = \lambda (G) x_r.
\end{cases}
\end{equation*}
Thus $\lambda (G)$ is the largest eigenvalue of
the following matrix $A_r$ corresponding to
eigenvector $(x_v,x_w,x_u$, $x_1^-,x_2^-, x_3 ,\ldots ,x_r)$,
where $A_r (r\ge 3)$ is defined as
\[  A_r:=
\left[ \begin{array}{ccccc;{2pt/2pt}ccc}
0 & 0 & 1 & 0 & b_2-1 & b_3 & \cdots & b_r \\
0&0&1&b_1-1 &0 & b_3 & \cdots & b_r \\
1&1&0&0&0&b_3 & \cdots & b_r \\
0&1&0&0 & b_2-1 & b_3 & \cdots & b_r \\
1 & 0&0& b_1-1 & 0 & b_3 & \cdots & b_r  \\
\hdashline[2pt/2pt]
1&1&1&b_1-1 & b_2-1 & 0 & \cdots & b_r \\
\vdots & \vdots & \vdots & \vdots & \vdots & \vdots & & \vdots \\
1&1&1&b_1-1 & b_2-1 & b_3 & \cdots & 0
\end{array} \right].  \]
For notational convenience, we denote
\[  A_2:=
\begin{bmatrix}
0 & 0 & 1 & 0 & b_2-1  \\
0&0&1&b_1-1 &0 \\
1&1&0&0&0  \\
0&1&0&0 & b_2-1  \\
1 & 0&0& b_1-1 & 0
\end{bmatrix}, \]
and
\[ R_{b_1,b_2}(x):=\det \begin{bmatrix}
x+1&1 & 0 & b_1-1 &0 \\
1&x+1&0&0&b_2-1\\
0&0&x+1 &b_1-1&b_2-1 \\
1&0&1&x+b_1-1&0\\
0&1&1&0&x+b_2-1
\end{bmatrix}. \]
For every $r\ge 2$, the characteristic polynomial of $A_r$ is denoted by
\[ F_{b_1,b_2,\ldots ,b_r}(x) =  \det (xI_{r+3} - A_r).\]
In particular, the polynomial $F_{b_1,b_2}(x)$
is the same as that in Lemma \ref{lempp}.
By expanding the last column of $\det (xI_{r+3}-A_r)$,
 we get the following recurrence relations:
\begin{equation} \label{eq5}
F_{b_1,b_2,b_3}(x) = (x+b_3) F_{b_1,b_2}(x) - b_3 R_{b_1,b_2}(x),
\end{equation}
and for every integer $r\ge 4$,
\begin{equation} \label{eq6}
F_{b_1,b_2,\ldots ,b_r}(x) =
(x+b_r) F_{b_1,b_2,\ldots ,b_{r-1}}(x)
- b_r\prod_{i=3}^{r-1} (x+b_i) R_{b_1,b_2}(x),
\end{equation}
where $F_{b_1,b_2}(x) $ and $R_{b_1,b_2}(x)$
are computed as below:
\begin{equation*}
\begin{aligned}
 F_{b_1,b_2}(x)
&= x^5 - (b_1b_2+1) x^3+ (3b_1b_2 -2b_1 -2b_2 +1)x
-2b_1b_2 +2b_1 +2b_2 -2,  \\
R_{b_1,b_2}(x)
&= x^5+(b_1+b_2+1)x^4 + (b_1b_2+1)x^3-(b_1b_2+b_1+b_2-3)x^2 \\
& \quad + (2b_1 + 2b_2 -3b_1b_2-1)x + 3(b_1-1)(b_2-1).
\end{aligned}
\end{equation*}
Note that $b_1\le b_2-2$. Upon computations, we obtain
\begin{equation*}
F_{b_1+1,b_2-1}(x) - F_{b_1,b_2}(x) = (b_1 -b_2+1)(x-1)^2(x+2) < 0,
\end{equation*}
and
\begin{equation*}
R_{b_1+1,b_2-1}(x) - R_{b_1,b_2}(x) = - (b_1-b_2+1)(x-1)(x^2-3) > 0.
\end{equation*}
Note that $b_1-b_2+1\le -1$.
Combining with equation (\ref{eq5}), we obtain
\begin{align*}
&F_{b_1+1,b_2-1,b_3}(x) - F_{b_1,b_2,b_3}(x)  \\
&= (x+b_3)(F_{b_1+1,b_2-1}(x)- F_{b_1,b_2}(x)) - b_3 ( R_{b_1+1,b_2-1}(x)- R_{b_1,b_2}(x) ) \\
&= (b_1-b_2+1)(x-1)^2(x+2)(x+b_3) + b_3 (b_1-b_2+1)(x-1)(x^2-3) \\
&= (b_1-b_2+1)(x-1)
\bigl( (x-1)(x+2)(x+b_3) +b_3 (x^2-3) \bigr) < 0.
\end{align*}
Next we prove by induction that for every $r\ge 3$ and $x\ge 2$,
\begin{equation} \label{eqcase2}
\begin{aligned}
F_{b_1+1,b_2-1,b_3,\ldots ,b_r}(x)
- F_{b_1,b_2,b_3,\ldots ,b_r}(x)  <0.
\end{aligned}
\end{equation}
Firstly, the base case $r=3$ was verified in the above.
For  $r\ge 4$, we get from (\ref{eq6}) that
\begin{align*}
&F_{b_1+1,b_2-1,b_3,\ldots ,b_r}(x)
- F_{b_1,b_2,b_3,\ldots ,b_r}(x) \\
&=(x+b_r) \bigl( F_{b_1+1,b_2-1,b_3,\ldots ,b_{r-1}}(x)- F_{b_1,b_2,b_3,\ldots ,b_{r-1}}(x) \bigr) \\
&\quad -b_r\prod_{i=3}^{r-1} (x+b_i) \bigl( R_{b_1+1,b_2-1}(x) - R_{b_1,b_2}(x) \bigr) <0,
\end{align*}
where the last inequality holds by applying inductive hypothesis
on the case $r-1$ and invoking
the fact $R_{b_1+1,b_2-1}(x) - R_{b_1,b_2}(x) >0$.
From inequality (\ref{eqcase2}), we know that
\[ F_{b_1+1,b_2-1,b_3,\ldots ,b_r}( \lambda (G))
< F_{b_1,b_2,b_3,\ldots ,b_r}(\lambda (G))=0.\]
Since $\lambda (G')$ is the largest root of
 $F_{b_1+1,b_2-1,b_3,\ldots ,b_r}(x)$,
this implies $\lambda (G) < \lambda (G')$.

{\bf Case C.}
Thirdly, we consider the case $b_1\le b_i-2$ for some
$i\in \{3,\ldots ,r\}$. We may assume by symmetry that $b_1\le b_3-2$.
Our treatment in this case is similar with that of Case (B).
 Let $G^*$ be the graph  obtained from $G$
  by deleting a vertex of $B_3$ with its incident edges,
  and add a new vertex to $B_1^-$ and connect this new vertex to
  all remaining vertices of $B_3$ and all vertices of $B_2\cup B_4 \cup \cdots \cup B_r$.
  We will prove that $\lambda (G) < \lambda (G^*)$.
 By Case (B), we may assume that $|b_1-b_2|\le 1$.
  Clearly, $\lambda (G^*)$
 is the largest root of $F_{b_1+1,b_2,b_3-1,b_4,\ldots ,b_r}(x)$.
First of all, we will show that
\begin{equation} \label{eqcase3}
F_{b_1+1,b_2,b_3-1}(x) - F_{b_1,b_2,b_3}(x) <0,
\end{equation}
and then by applying induction, we will prove that for each $r\ge 4$,
\begin{equation} \label{eqcase3second}
 F_{b_1+1,b_2,b_3-1,b_4,\ldots ,b_r}(x)
- F_{b_1,b_2,b_3,b_4,\ldots ,b_r}(x) <0.
\end{equation}
Next, we verify  inequalities (\ref{eqcase3}) and  (\ref{eqcase3second}) for the case $r=4$ only,
since the inductive steps are the same as that of Case (B)
with slight differences.
By computation, we obtain
\begin{align*}
&(x+b_3-1)F_{b_1+1,b_2}(x)- (x+b_3)F_{b_1,b_2}(x)\\
&= -x^5 -b_2x^4 +(b_2(b_1-b_3+1)+1)x^3+ (3b_2-2)x^2 \\
&\quad + (3b_2b_3-3b_1b_2+2b_1-3b_2-2b_3+3)x
+ 2b_1b_2 - 2b_1 - 2b_2b_3+2b_3,
\end{align*}
and
\begin{align*}
&-(b_3-1)R_{b_1+1,b_2}(x) +b_3R_{b_1,b_2}(x) \\
&=x^5 + (b_2+b_1-b_3+2)x^4 +
(b_2(b_1-b_3+1)+1)x^3 \\
& \quad +(-b_1b_2-b_1+b_2b_3-2b_2+b_3+ 2)x^2 \\
& \quad +(3b_2b_3-3b_1b_2+2b_1-b_2-2b_3+1)x  \\
& \quad + 3b_1b_2-3b_1-3b_2b_3+ 3b_3.
\end{align*}
Combining these two equations with (\ref{eq5}), we get
\begin{align*}
&F_{b_1+1,b_2,b_3-1}(x) - F_{b_1,b_2,b_3}(x) \\
&=(x+b_3-1)F_{b_1+1,b_2}(x)-(x+b_3)F_{b_1,b_2}(x) - (b_3-1)R_{b_1+1,b_2}(x) + b_3 R_{b_1,b_2}(x) \\
&=(b_1-b_3+2)x^4 + 2(b_2(b_1-b_3+1)+1)x^3 + (b_2b_3-b_1b_2-b_1+b_2+b_3)x^2 \\
&\quad +(6b_2b_3-6b_1b_2 + 4b_1 - 4b_2 - 4b_3 + 4)x + 5b_1b_2-5b_1-5b_2b_3+5b_3 .
\end{align*}
Combining $|b_1-b_2|\le 1$ and $b_1- b_3\le -2$,
 one can verify that $F_{b_1+1,b_2,b_3-1}(x) < F_{b_1,b_2,b_3}(x)$
for every $x\ge 2(b_1-2)$.
This completes the proof of (\ref{eqcase3}).
We now consider  (\ref{eqcase3second}) in the case $r=4$. Note that
$b_1- b_3+2\le 0$ and
\begin{align*}
&-(x+b_3-1)R_{b_1+1,b_2}(x) + (x+b_3)R_{b_1,b_2}(x) \\
&=(b_1-b_3+2)x^4 + (b_2(b_1 -b_3+ 2)+2)x^3 + (b_2(b_3-b_1+1)+b_3-b_1)x^2 \\
& \quad  +(3b_2b_3 -3b_1b_2+ 2b_1- 4b_2- 2b_3+4)x + 3b_1b_2- 3b_1- 3b_2b_3+ 3b_3 <0,
\end{align*}
which together with  (\ref{eq6}) and the case $r=3$ yields
\begin{align*}
&F_{b_1+1,b_2,b_3-1,b_4}(x) - F_{b_1,b_2,b_3,b_4}(x) \\
& = (x+b_4) (F_{b_1+1,b_2,b_3-1}(x)- F_{b_1,b_2,b_3}(x) )  \\
 & \quad - b_4(x+b_3-1)R_{b_1+1,b_2}(x) + b_4(x+b_3)R_{b_1,b_2}(x)<0.
\end{align*}
Let $t=\min\{b_i: 1\le i\le r\} -1$.
Since the complete $r$-partite $K_{t,t,\ldots ,t}$
is a subgraph of $G$, we know that  $\lambda (G) \ge
\lambda (K_{t,t,\ldots ,t})=(r-1)t$.
Thus, we can  get $F_{b_1+1,b_2,b_3-1,b_4, \ldots ,b_r}
(\lambda (G))
< F_{b_1,b_2,b_3,b_4,\ldots ,b_r}(\lambda (G))=0$,
which yields $\lambda (G) < \lambda (G^*)$, which is a
contradiction.

{\bf Case D.}
Finally, we consider the case  $b_i \le b_1 -1$ for some $i\ge 3$.
We may assume that $b_3\le b_1-1$.
This case could  be completed by applying a similar argument of Case (C). Let $G^*$ be the graph obtained from $G$
by removing a vertex of $B_1^-$ with its incident edges, and
adding a copy of a vertex of $B_3$.
In what follows, we will show that
\begin{equation}  \label{eqcase4}
 F_{b_1-1,b_2,b_3+1}(x) - F_{b_1,b_2,b_3}(x)<0,
 \end{equation}
and then we prove by induction that for every $r\ge 4$,
\begin{equation}  \label{eqcase4second}
 F_{b_1-1,b_2,b_3+1,b_4,\ldots ,b_r}(x)
 - F_{b_1,b_2,b_3,b_4,\ldots ,b_r}(x)<0.
 \end{equation}
By computation, we obtain that
\begin{align*}
&(x+b_3+1)F_{b_1-1,b_2}(x) - (x+b_3)F_{b_1,b_2}(x) \\
&= x^5 + b_2x^4 + ( b_2(b_3-b_1+1)-1)x^3 + (-3b_2+2)x^2 \\
& \quad + (3b_1b_2- 2b_1- 3b_2b_3- 3b_2+ 2b_3+1)x - 2b_1b_2+ 2b_1+ 2b_2b_3+ 4b_2- 2b_3-4,
\end{align*}
and
\begin{align*}
&- (b_3+1) R_{b_1-1,b_2}(x) + b_3 R_{b_1,b_2}(x) \\
&= -x^5 + (b_3- b_1- b_2)x^4 + (b_2(b_3-b_1+1)-1)x^3 \\
&\quad + (b_1b_2-b_2b_3+b_1-b_3-4)x^2 \\
&\quad +(3b_1b_2- 2b_1-3b_2b_3- 5b_2+ 2b_3+3)x \\
& \quad  +3b_2b_3- 3b_1b_2+ 3b_1+6b_2-3b_3-6.
\end{align*}
Combining with the recurrence equation (\ref{eq5}), we get
\begin{align*}
&F_{b_1-1,b_2,b_3+1}(x) - F_{b_1,b_2,b_3}(x) \\
&= (x+b_3+1)F_{b_1-1,b_2}(x) - (x+b_3)F_{b_1,b_2}(x)  - (b_3+1) R_{b_1-1,b_2}(x) + b_3 R_{b_1,b_2}(x) \\
&= (b_3-b_1)x^4 + ( 2b_2(b_3-b_1+1)-2)x^3 + (b_1b_2-b_2b_3+b_1-b_3-3b_2-2)x^2 \\
&\quad + (6b_1b_2-6b_2b_3- 4b_1- 8b_2+ 4b_3+4)x - 5b_1b_2+ 5b_1+5b_2b_3+10b_2-5b_3-10.
\end{align*}
Since $b_3-b_1\le -1$ and $|b_1-b_2|\le 1$,
one can verify that $F_{b_1-1,b_2,b_3+1}(x) - F_{b_1,b_2,b_3}(x)<0$
for every $x\ge 2(b_3-1)$.  This completes the proof of (\ref{eqcase4}).
Next we will prove (\ref{eqcase4second}) only for the case $r=4$,
since the inductive steps are similar with that of Cases (B) and (C).
By computation, we have
\begin{align*}
&-(x+b_3+1)R_{b_1-1,b_2}(x) + (x+b_3)R_{b_1,b_2}(x) \\
&=(b_3-b_1)x^4 + (b_2(b_3-b_1)-2)x^3 + (b_1b_2+b_1- b_2b_3- 3b_2- b_3- 2)x^2 \\
&\quad + (3b_1b_2- 2b_1- 3b_2b_3- 2b_2+ 2b_3)x  - 3b_1b_2 + 3b_1+ 3b_2b_3+ 6b_2- 3b_3-6<0,
\end{align*}
which together with  (\ref{eq6}) and the case $r=3$ gives
 \begin{align*}
 &F_{b_1-1,b_2,b_3+1,b_4}(x) - F_{b_1,b_2,b_3,b_4}(x) \\
&= (x+b_4)(F_{b_1-1,b_2,b_3+1}(x) - F_{b_1,b_2,b_3}(x) ) \\
&\quad -b_4(x+b_3+1)R_{b_1-1,b_2}(x) + b_4(x+b_3)R_{b_1,b_2}(x)<0.
 \end{align*}
Since $F_{b_1-1,b_2,b_3+1,b_4, \ldots ,b_r}
(\lambda (G))
< F_{b_1,b_2,b_3,b_4,\ldots ,b_r}(\lambda (G))=0$
and $\lambda (G^*)$ is the largest root of
 $F_{b_1-1,b_2,b_3+1,b_4,\ldots ,b_r}(x)$,
we know that $\lambda (G) < \lambda (G^*)$, which
contradicts with the choice of $G$.
In summary, we complete the proof of all possible cases.
\end{proof}

\noindent
{\bf Remark.}
It seems possible to prove the last three cases by
using a weight-balanced argument similar with that of the first case.
Nevertheless,
it is inevitable that a great deal of tedious calculations are required
in the proof of these cases.
Moreover, applying the recursive technique of determinants in the proof of Lemma \ref{lem42}, one can compute the characteristic
polynomial of the adjacency matrix and signless Laplacian matrix
of the $n$-vertex complete $r$-partite graph $K_{t_1,\ldots ,t_r}$. More precisely,
\begin{align*}
\det (xI_n-A(K_{t_1,\ldots ,t_r}))
 = x^{n-r} \left( 1-\sum_{i=1}^r \frac{t_i}{x+t_i}\right) \prod_{i=1}^r(x +t_i),
 \end{align*}
and
\begin{align*}
 \det (xI_n-Q(K_{t_1,\ldots ,t_r})) =  \prod_{i=1}^r(x-n+t_i)^{t_i-1} (x-n+2t_i)
\left( 1-\sum_{i=1}^r \frac{t_i}{x-n+2t_i}\right).
\end{align*}
It has its own interests to compute the eigenvalues of
complete multipartite graphs;
see, e.g.,  \cite{Del2012,YWS2011,Obo2019,YYSW2019} for different proofs
and related results.

\medskip
Now, we are ready to give the proof of Theorem \ref{thm214-restate}.

\begin{proof}[{\bf Proof of Theorem \ref{thm214-restate}}]
Assume that $G$ is a $K_{r+1}$-free  non-$r$-partite graph on $n$ vertices
with maximum value of the spectral radius.
Our goal is to prove that $G=Y_r(n)$.
Clearly,  $G$ must be a connected graph.
Let $\bm{x} \in \mathbb{R}^n$ be a positive unit eigenvector of $\lambda (G)$.

\begin{claim} \label{claim4.1}
There exists a vertex $u\in V(G)$ such that
$G\setminus \{u\}$ is $r$-partite.
\end{claim}

\begin{proof}[Proof of Claim \ref{claim4.1}]
Recall in Section \ref{sec3} that for two non-adjacent vertices $u,v\in V(G)$,
 the spectral Zykov symmetrization $Z_{u,v}(G)$
is defined as the graph obtained from $G$ by removing all edges incident to  $u$ and then adding new edges from $u$ to $N_G(v)$.
We can verify that the spectral Zykov symmetrization does not
increase the clique number and the chromatic number.
 Recall that $s_G(v,\bm{x}) =\sum_{i\in N_G(v)} x_i$ is
 the sum of weights of
all neighbors of $v$ in $G$.
For two non-adjacent vertices $u,v$, if $s_G(u,\bm{x}) < s_G(v,\bm{x})$,
then we replace $G$ with $Z_{u,v}(G)$.
If $s_G(u,\bm{x}) = s_G(v,\bm{x})$, then we can apply either $Z_{u,v}$ or $Z_{v,u}$,
which leads to $N(u)=N(v)$ after making the spectral Zykov symmetrization.
Obviously,
the spectral Zykov symmetrization does not create a copy of $K_{r+1}$.
More significantly, it will increase  the spectral radius strictly,
since $\bm{x}$ is entry-wise positive.

The  proof of Claim \ref{claim4.1} is based on the spectral Zykov symmetrization stated in above.
Since $G$ is $K_{r+1}$-free, we can  repeatedly apply
the spectral Zykov symmetrization on every pair of non-adjacent vertices
until $G$ becomes an $r$-partite graph.
Without loss of generality, we may assume that
$G$ is $K_{r+1}$-free and $G$ is not $r$-partite,
while $Z_{u,v}(G)$ is $r$-partite.
Thus $G \setminus \{u\}$ is $r$-partite,
and we assume that  $V(G) \setminus \{u\} :=V_1\cup V_2 \cup \cdots \cup V_r$ is a $r$-partition,
where $V_1,V_2,\ldots ,V_r$ are pairwise disjoint and $\sum_{i=1}^r |V_i| =n-1$.
\end{proof}

We denote $A_i=N(u)\cap V_i$ for every $i\in [r]:=\{1,\ldots ,r\}$.
Note that $G$ has maximum spectral radius among all $K_{r+1}$-free non-$r$-partite graphs. Then for each $i\in [r]$,
every vertex of $ V_i \setminus A_i$ is adjacent to every vertex of $V_j$
for every $j\in [r]$ and $j\neq i$.
We remark here that the difference between
the  $K_{r+1}$-free case (Theorem \ref{thm214}) and
the  triangle-free case (Theorem \ref{thmLNW})  is that
there may exist some edges between the pair of sets $A_i$ and $A_j$,
which makes the problem seem more difficult.

\begin{claim} \label{claim4.2}
There exists a pair $\{i,j\} \subseteq [r]$ such that
$G[A_i,A_j]$ forms an empty graph, and for other pairs
  $\{s,t\} \neq \{i,j\}$,
 $G[A_s,A_t]$ is a complete bipartite subgraph in $G$.
  \end{claim}

\begin{proof}[Proof of Claim \ref{claim4.2}]
Let $G[A_1,A_2,\ldots ,A_r]$ be the subgraph of $G$ induced by the vertex sets $A_1,A_2$, $\ldots ,A_r$. Note by Claim \ref{claim4.1}
that $G[A_i]$ has no edge.
Claim \ref{claim4.2} is equivalent to say that
$G[A_1\cup A_2,A_3,\ldots ,A_r]$ forms a complete $(r-1)$-partite subgraph in $G$.
Since $G$ is $K_{r+1}$-free, we know that
the subgraph $G[A_1,A_2,\ldots ,A_r]$
is a $K_{r}$-free subgraph in $G$.

First of all, we choose a vertex $v_1\in A_1$
such that $s_G(v_1,\bm{x})$ is the maximum among all vertices of $ A_1$,
and observe that any two vertices of $A_1$ are not adjacent, then we apply the spectral Zykov operation $Z_{w,v_1}$
on $G$ for every vertex $w\in A_1\setminus \{v_1\}$.
These operations will make all vertices of $A_1$ being equivalent,
that is, every pair of vertices in $A_1$ has the same neighbors.
Secondly, we choose a vertex $v_2\in A_2$
such that $s_G(v_2,\bm{x})$ is maximum over all vertices of $ A_2$,
and then we apply similarly the Zykov operation $Z_{w,v_2}$
on $G$ for every $w\in A_2\setminus \{v_2\}$.
{\it Note  that all vertices in $A_1$ have the same neighbors.}
After doing Zykov's operations on vertices of $A_2$,
we claim that the induced  subgraph $G[A_1,A_2]$
is either a complete bipartite graph or an empty graph.
Indeed, if $v_2\in \cap_{v\in A_1}N(v)$, then
the operations $Z_{w,v_2}$ for all $w\in A_2 \setminus \{v_2\}$
 will lead to a complete bipartite graph between $A_1$ and $A_2$.
If $v_2\notin \cap_{v\in A_1}N(v)$, then $v_2$ is not adjacent to
all vertices of $A_1$, and so is $w$ for every $w\in A_2 \setminus \{v_2\}$, which yields that
$G[A_1,A_2]$ is an empty graph.
Moreover, by applying the similar operations on $A_3,A_4, \ldots ,A_r$,
we can obtain that for every $i,j\in [r]$ with $i\neq j$,
the induced bipartite subgraph $G[A_i,A_j]$ is either  complete bipartite or empty.
Since $G[A_1,A_2,\ldots ,A_r]$ is $K_r$-free and $G$ attains the maximum spectral radius, we know that there is exactly one pair $\{i,j\}\subseteq [r]$
such that $G[A_i,A_j]$ is an empty graph.
 \end{proof}

 We may assume that $\{i,j\}=\{1,2\}$.
In what follows,
we intend to enlarge  $ A_i$ to the whole set $V_i$ for every $i\in \{3,4,\ldots ,r\}$.
 Observe that every vertex of $V_i \setminus A_i$
 is adjacent to every vertex of $V_j$ for each $j\in [r]$ with $j\neq i$,
 adding all edges between $\{u\}$ and  $V_i \setminus A_i$
 does not create a copy of $K_{r+1}$, and it increase the spectral radius of $G$ strictly.
This observation implies that $u$ is adjacent to every vertex of $V_i$ for each $i\in \{3,4,\ldots ,r\}$;  see  (a)  in Figure \ref{fig-6}.

\begin{figure}[htbp]
\centering
\includegraphics[scale=0.7]{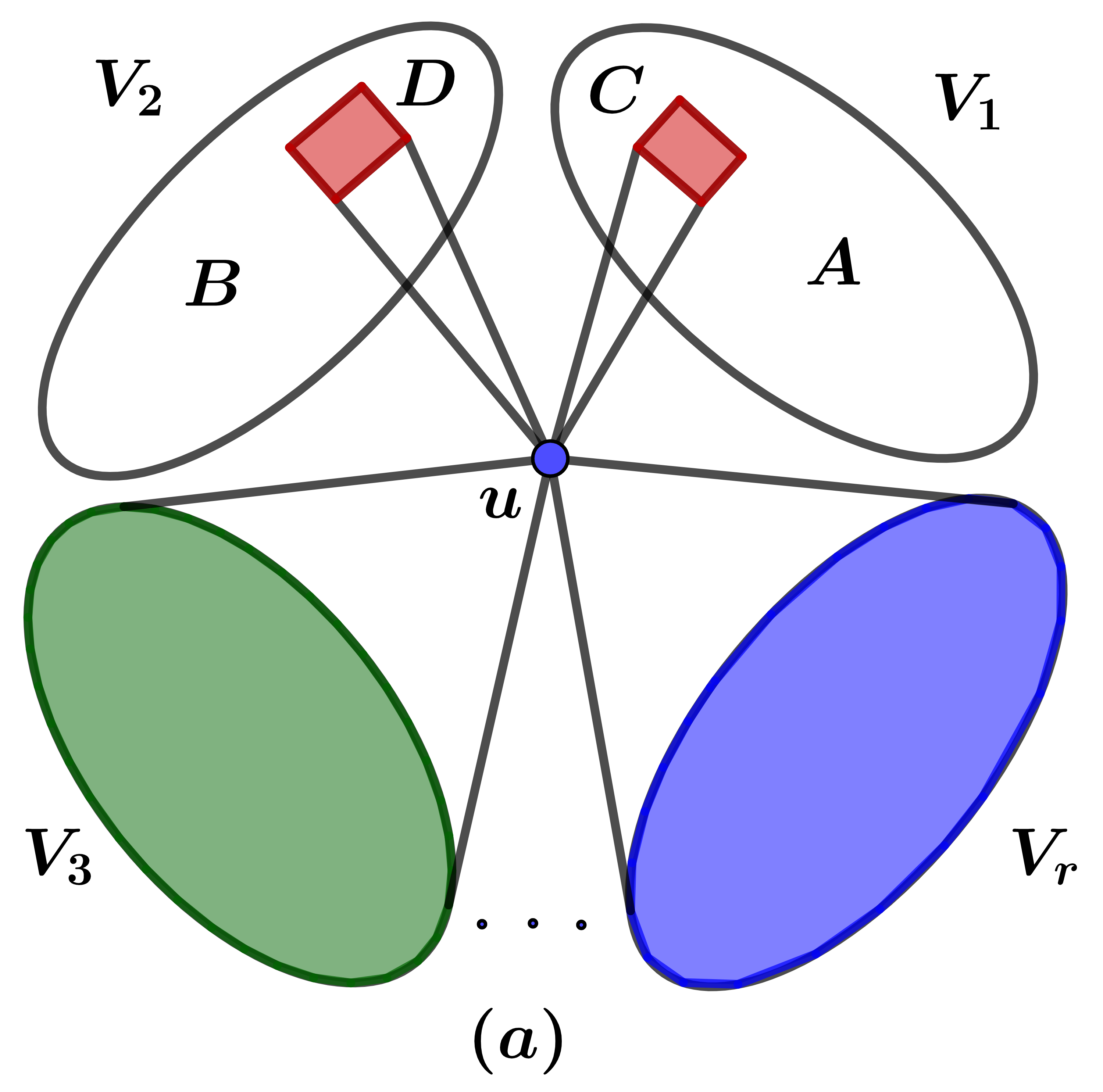}
\quad  \includegraphics[scale=0.7]{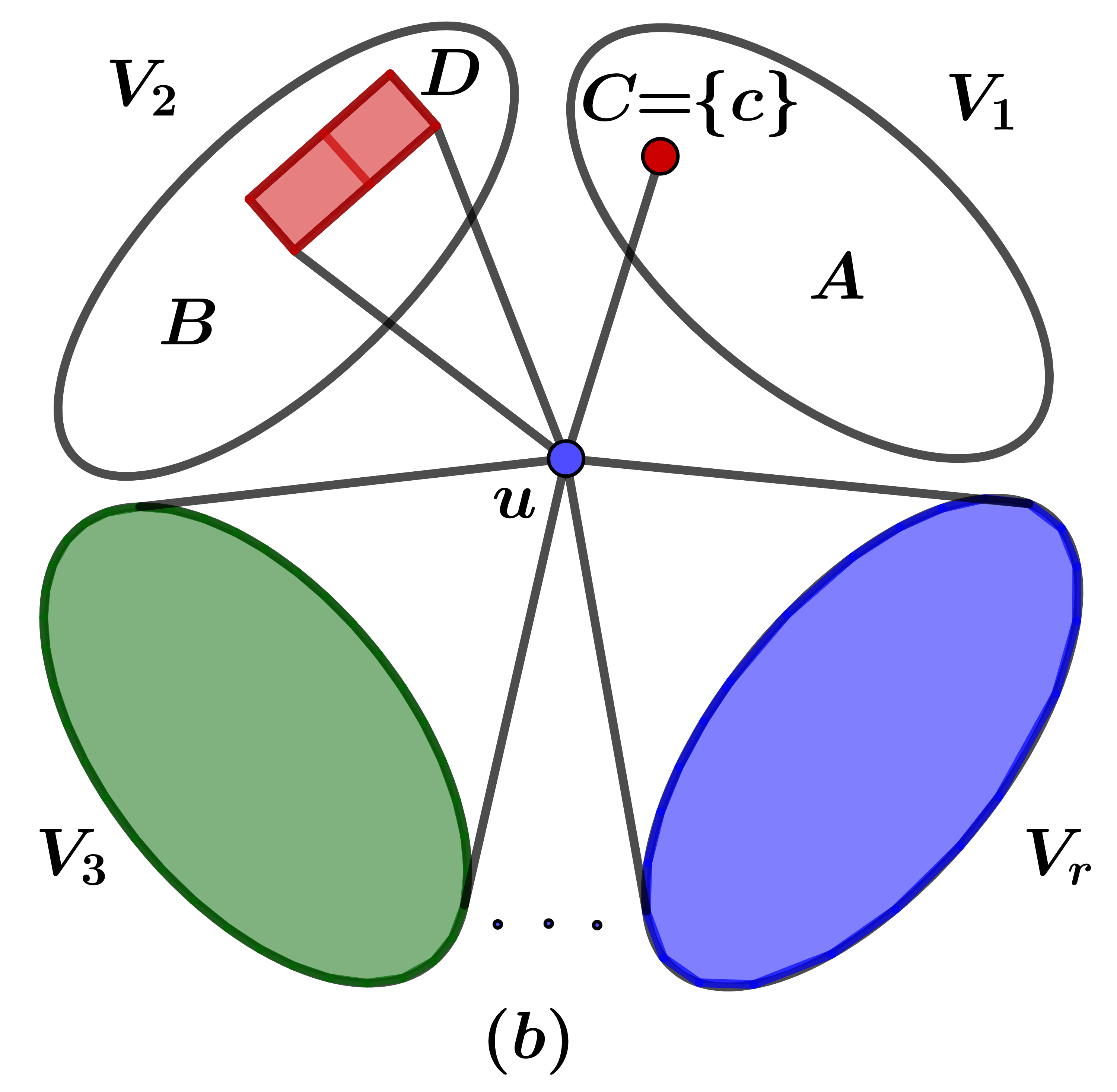}
\quad \includegraphics[scale=0.7]{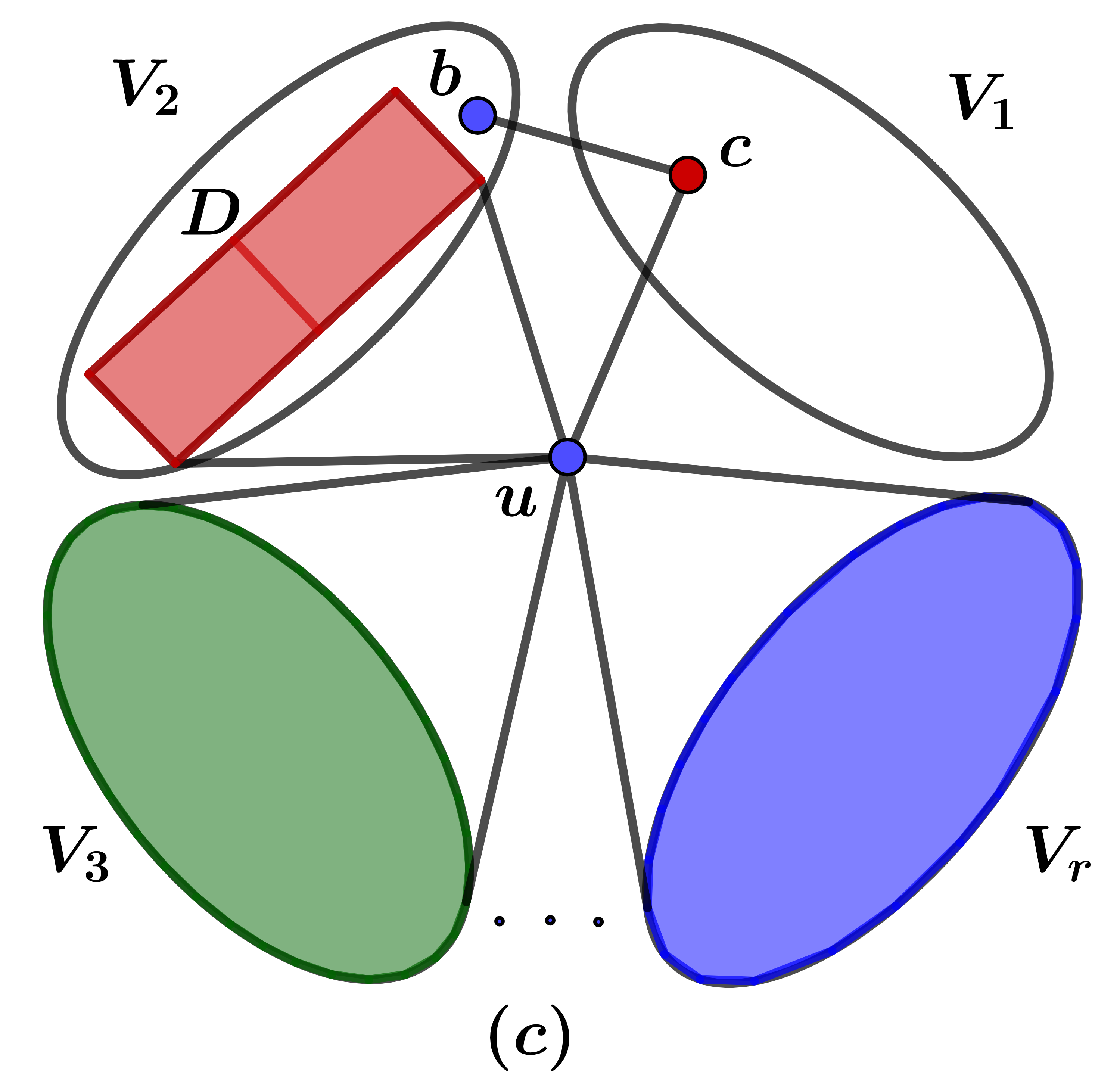}
\caption{Local changes and switches.}
    \label{fig-6}
\end{figure}

Assume that $C:=N(u)\cap V_1$ and $D:=N(u)\cap V_2$. We denote
$A:=V_1 \setminus C$ and $B:=V_2\setminus D$; see (a) in Figure \ref{fig-6}.
Note that there is no edge between $C$ and $D$.
In the remaining of the proof,
we will prove  that both $C$ and $D$ are single vertex sets.
The treatment is similar with that
of our proof of Theorem \ref{thmLNW}.

\begin{claim} \label{claim4.3}
One of the sets $C$ or $D$ has size $1$.
\end{claim}

\begin{proof}[Proof of Claim \ref{claim4.3}]
If $\sum_{v\in A}x_v \ge \sum_{v\in B}x_v$, then we choose $|C|-1$ vertices of $C$ and delete its incident edges only in $B$,
then we move these $|C|-1$ vertices into $D$ and
connect these vertices to $A$. In this process,
the edges between these $|C|-1$ vertices  and $V_3 \cup \cdots \cup V_r$ are unchanged.
We write $G'$ for the resulting graph.
Using the similar computation as in Section \ref{sec3},
 we can verify that  $\lambda (G') > \lambda (G)$.

If $\sum_{v\in A} x_v < \sum_{v\in B} x_v$,
then we can choose $|D|-1$ vertices of $D$  and
delete its incident edges only in $A$,
and then move these $|D|-1$ vertices into $C$ and
join these vertices to $B$.
This process will increase strictly the spectral radius.
From the above case analysis, we can always remove
the vertices of $G$ to force
either $|C|=1$ or $|D|=1$. \end{proof}

We may assume by symmetry that $|C|=1$ and
denote $C=\{c\}$;
see (b) in Figure \ref{fig-6}.

 \begin{claim} \label{claim4.4}
The set $D$ is a single vertex, i.e., $|D|=1$.
\end{claim}

 \begin{proof}[Proof of Claim \ref{claim4.4}]
If $x_u< x_c$, then we choose $|D|-1$ vertices of $D$
and delete  its incident edges to vertex $u$,
then we move these $|D|-1$ vertices into
$B$ and join these these vertices to $c$, and
keeping the other edges unchanged, we denote the new graph by $G^{\star}$. Then
we can similarly get $\lambda (G^{\star}) > \lambda (G)$.
In the graph $G^{\star}$, we have $|D|=1$ and write $D=\{d\}$.
Thus $G^{\star}$ is the graph obtained from a complete $r$-partite  graph
$K_{t_1,t_2,\ldots ,t_r}$, where $\sum_{i=1}^r t_i=n-1$, by adding a new vertex $u$ and then joining $u$ to a vertex $c\in V_1$, and joining $u$ to a vertex $d\in V_2$,
and joining $u$ to all vertices of $V_3 \cup \cdots \cup V_r$, and
finally removing the edge $cd\in E(K_{t_1,t_2,\ldots ,t_r})$.

If $x_u\ge x_c$, then we choose $|B|-1$ vertices of $B$
and delete  its incident edges to vertex $c$,
then we move these $|B|-1$ vertices into $D$ and
join these vertices to vertex $u$.
We  denote the new graph by
$G^*$. Then $\lambda (G^*) > \lambda (G)$.
Thus we conclude  in the new graph $G^*$ that $B$ is a single vertex,
say $B=\{b\}$; see (c) in Figure \ref{fig-6}.
In what follows, we will exchange the position of $u$ and $c$.
Note that $c\in V_1$ is adjacent to a vertex $b\in V_1$ and
all vertices of $V_3 \cup \cdots \cup V_r$.
Now, we move vertex $c$ outside of $V_1$ and put vertex $u$ into $V_1$.
Thus the new center $c$ is adjacent to a vertex $u\in V_1$, a vertex $b\in V_2$
and all vertices of $V_3 \cup \cdots \cup V_r$. Note that $bu\notin E(G^*)$.
Hence $G^*$
has the same structure as the previous case,
and then we may assume that $|D|=1$.
\end{proof}

From the above discussion,
we know that $G$ is isomorphic to the graph
defined as in  Lemma \ref{lem42}.
By applying  Lemma \ref{lem42},
we obtain that  $\lambda (G) \le \lambda (Y_r(n))$.
Moreover, the equality holds if and only if $G=Y_r(n)$. This completes the proof of Theorem \ref{thm214-restate}.
\end{proof}

\section{Unified extension to the $p$-spectral radius}

\label{sec5}

Recall that the spectral radius of a graph is defined
as the largest  eigenvalue of its adjacency matrix.
By Rayleigh's theorem, we know that
it is also equal to the maximum value of
$\bm{x}^TA(G)\bm{x}=2\sum_{\{i,j\}\in E(G)} x_ix_j$ over all
$\bm{x}\in \mathbb{R}^n$ with $|x_1|^2 + \cdots +|x_n|^2=1$.
The definition of the spectral radius was recently extended
to the $p$-spectral radius; see \cite{KLM2014, KN2014} and references therein.
We denote the $p$-norm of $\bm{x}$ by
$\lVert \bm{x}\rVert_p
=( |x_1|^p + \cdots +|x_n|^p)^{1/p}$.
For every real number $p\ge 1$,
the {\it $p$-spectral radius} of $G$ is defined as
\begin{equation*} \label{psp}
  \lambda^{(p)} (G) : = 2 \max_{\lVert \bm{x}\rVert_p =1}
 \sum_{\{i,j\} \in E(G)} x_ix_j.
\end{equation*}
We remark that $\lambda^{(p)}(G)$ is a versatile parameter.
Indeed, $\lambda^{(1)}(G)$ is known as the Lagrangian function of $G$,
$\lambda^{(2)}(G)$ is the spectral radius of its adjacency matrix,
and
\begin{equation}  \label{eqlimit}
 \lim_{p\to +\infty} \lambda^{(p)} (G)=2e(G),
 \end{equation}
 which can be guaranteed by
$
  2e(G)n^{-2/p} \le \lambda^{(p)}(G) \le (2e(G))^{1-1/p}.
$
  To some extent, the $p$-spectral radius could be viewed as a unified extension of
  the   spectral radius as well as the size of a graph.
In addition, it is worth mentioning that
if $ 1\le q\le p$, then  $\lambda^{(p)}(G)n^{2/p} \le \lambda^{(q)}(G)n^{2/q}$
and $(\lambda^{(p)}(G)/2e(G))^p \le (\lambda^{(q)}(G)/2e(G))^q$;
see \cite[Propositions 2.13 and 2.14]{Niki2014laa}.
As commented by Kang and Nikiforov in \cite[p.3]{KN2014},
linear-algebraic methods are irrelevant for the study of
$\lambda^{(p)}(G)$ in general, and in fact no efficient methods
are known for it. Thus the study of $\lambda^{(p)}(G)$
for $p\neq 2$ is far more complicated than the classical
spectral radius.
In 2014, Kang and Nikiforov \cite{KN2014}
proved the following result for the $p$-spectral radius.

\begin{theorem}[Kang--Nikiforov, 2014] \label{thmKN-p}
If $G$ is a $K_{r+1}$-free graph on $n$ vertices,
then for every $p>1$,
\[  \lambda^{(p)}(G)\le \lambda^{(p)}(T_r(n)), \]
equality holds if and only if $G$ is the $n$-vertex Tur\'{a}n graph $T_r(n)$.
\end{theorem}

\noindent
{\bf Remark.}
We remark that
a theorem of Motzkin and Straus \cite{MS1965} states that
Theorem \ref{thmKN-p} is also valid for  $p=1$
except for the extremal graphs attaining the equality.
Keeping (\ref{eqlimit}) in mind, we can see that
Theorem \ref{thmKN-p} is a unified extension of
 both Tur\'{a}n's Theorem 
and spectral Tur\'{a}n's Theorem \ref{thm460} by taking $p\to +\infty$ and $p=2$, respectively.
\medskip

A vector $\bm{x}\in \mathbb{R}^n$
is called a unit eigenvector corresponding to $\lambda^{(p)} (G)$
if it satisfies $\sum_{i=1}^n |x_i|^p=1$
and $\lambda^{(p)} (G)=2\sum_{\{i,j\} \in E(G)} x_ix_j$.
By  Lagrange's multiplier method, there exists a positive unit eigenvector whenever $G$ is connected.
The proof of Theorem \ref{thm214-restate} relies on
the Rayleigh representation of $\lambda (G)$ and the existence
of a positive eigenvector.
For the $p$-spectral radius,
 there is also a positive vector
corresponding to $\lambda^{(p)}(G)$.
Applying the similar techniques,
one can  extend Theorem \ref{thm214-restate} to
the $p$-spectral radius. We leave the details for interested readers.

\begin{theorem}
 Let $G$ be an $n$-vertex graph.
If $G$ does not contain $K_{r+1}$ and $G$ is not $r$-partite,
then for every $p>1$, we have
\[  \lambda^{(p)} (G) \le \lambda^{(p)} (Y_r(n)).  \]
Moreover, the equality holds if and only if $G=Y_r(n)$.
\end{theorem}

\section{Concluding remarks}

\label{sec6}

We shall conclude with
some possible problems for interested readers.
To begin with, we define an extremal function $\psi (n,F,t)$
as the maximum number of edges in an $n$-vertex $F$-free graph with
the chromatic number $\chi (G) \ge t$.
In particular, Theorem \ref{thmBrouwer} says that
$\psi (n,K_{r+1}, r+1) = e(T_r(n)) - \lfloor \frac{n}{r} \rfloor +1$.
Similarly, we can define the spectral extremal function as
$ \psi_{\lambda} (n,F,t):= \max \{ \lambda (G) : |G| =n,
F\nsubseteq G , \chi (G)\ge t\}$.
In Theorem \ref{thm214}, we have proved that
$\psi_{\lambda} (n,K_{r+1}, r+1) = \lambda (Y_r(n))$.
It is meaningful to study
the functions $ \psi (n,F,t)$
and $ \psi_{\lambda} (n,F,t)$ in general.

We write $q(G)$ for the
  signless Laplacian spectral radius, i.e.,
 the largest eigenvalue of
 the {\it signless Laplacian matrix}  $Q(G)=D(G) +
 A(G)$, where $D(G)= \mathrm{diag} (d_1,\ldots ,d_n)$
 is the degree diagonal matrix and
 $A(G)$ is the adjacency matrix.
In 2013,
He, Jin and Zhang \cite[Theorem 1.3]{HJZ2013}  proved the signless Laplacian spectral version of Tur\'{a}n's theorem, which states that if $G$ is a $K_{r+1}$-free graph on
$n$ vertices, then $q(G)\le q(T_r(n))$, equality holds if and only if $r=2$ and $G=K_{t,n-t}$ for some $t$, or $r\ge 3$ and $G=T_r(n)$.
Similar with the adjacency spectral radius, the signless Laplacian spectral version also implies
the edge Tur\'{a}n theorem.
It is interesting to study whether it is possible to
 extend the results of our paper
in terms of $q(G)$.
For example, given an integer $r\ge 3$, whether $Y_r(n)$ is the extremal graph attaining
the maximum signless Laplacian spectral radius among all
non-$r$-partite $K_{r+1}$-free graphs.

\medskip

After the paper is submitted, we are aware of some recent results
 \cite{GLZ2021, ZZ2021} for non-bipartite $C_{2k+1}$-free graphs
with $k\ge 2$, and \cite{LMX2022} for the signless Laplacian spectral radius of non-bipartite $C_3$-free graphs.

\subsection*{Acknowledgements}
We are
thankful to the reviewers for reading the manuscript carefully, checking all the details
and giving insightful comments to help improve the manuscript. The research was supported by NSFC (Grant No.11931002).

\end{document}